\documentclass[12pt]{amsart}
\usepackage{amssymb,amscd}
\usepackage{tikz}

\input xy \xyoption{all}

\usepackage[colorlinks=true,allcolors = blue]{hyperref} % my new

\let\mc\mathcal
\usepackage{eucal}
\let\euc\mathcal
\let\mathcal\mc

\let\frak\mathfrak

\def\>{\relax\ifmmode\mskip.666667\thinmuskip\relax\else\kern.111111em\fi}
\def\<{\relax\ifmmode\mskip-.333333\thinmuskip\relax\else\kern-.0555556em\fi}
\def\:{\relax\ifmmode\mskip.333333\thinmuskip\relax\else\kern.0555556em\fi}
\def\?{\relax\ifmmode\mskip-.666667\thinmuskip\relax\else\kern-.111111em\fi}
\def\vsk#1>{\vskip#1\baselineskip} 
\def\vv#1>{\vadjust{\vsk#1>}\ignorespaces}
\def\vvn#1>{\vadjust{\nobreak\vsk#1>\nobreak}\ignorespaces}
\def\vvgood{\vadjust{\penalty-500}} \let\alb\allowbreak
\def\fratop{\genfrac{}{}{0pt}1}
\def\satop#1#2{\fratop{\scriptstyle#1}{\scriptstyle#2}}
\let\dsize\displaystyle  \let\ssize\scriptstyle
\let\sssize\scriptscriptstyle 
  \let\hp\hphantom
\def\dfrac{\dsize\frac}

\let\Smallskip\smallskip
\def\smallskip{\par\Smallskip}
\let\Medskip\medskip
\def\medskip{\par\Medskip}
\let\Bigskip\bigskip
\def\bigskip{\par\Bigskip}

\let\Maketitle\maketitle
\def\maketitle{\Maketitle\thispagestyle{empty}\let\maketitle\empty}

\newtheorem{thm}{Theorem}[section]
\newtheorem{cor}[thm]{Corollary}
\newtheorem{lem}[thm]{Lemma}
\newtheorem{prop}[thm]{Proposition}

\newcommand{\dontprint}[1]\relax
\numberwithin{equation}{section}

\theoremstyle{definition}
\newtheorem*{rem}{Remark}
\newtheorem*{example}{Example}
\newtheorem*{exam}{Example \exno}
\newtheorem*{defn}{Definition}

\let\nc\newcommand

\let\al\alpha
\let\bt\beta
\let\dl\delta
\let\Dl\Delta
\let\epsi\epsilon
\let\eps\varepsilon

\let\gm\gamma
\let\Gm\Gamma
\let\ka\kappa
\let\la\lambda
\let\La\Lambda
\let\om\omega
\let\Om\Omega
\let\pho\phi
\let\phi\varphi
\let\si\sigma

\let\thi\vartheta

\let\zt\zeta

\let\der\partial
\let\Hat\widehat

\let\ox\otimes
\let\Tilde\widetilde
\let\bra\langle
\let\ket\rangle

\let\ge\geqslant
\let\geq\geqslant
\let\le\leqslant
\let\leq\leqslant

\let\on\operatorname
\let\bi\bibitem
\let\bs\boldsymbol

\let\Empty\varnothing

\def\C{{\mathbb C}}

\def\Z{{\mathbb Z}}

\def\Ne{{\euc N}}
\def\Ue{{\euc U}}

\def\LL{{\mc L}}

\def\B{{\mc B}}

\def\F{{\mc F}}

\def\Rc{{\mc R}}

\def\+#1{^{\{#1\}}}
\def\lsym#1{#1\alb\dots\relax#1\alb}
\def\lc{\lsym,}

\def\id{{\on{id}}}

\def\Re{\on{Re}}

\def\tbigoplus{\mathop{\textstyle{\bigoplus}}\limits}
\def\tbigcup{\mathop{\textstyle{\bigcup}}\limits}

\def\pti{\textit{pt}}
\def\hor{\textrm{hor}}
\def\vert{\textrm{vert}}
\def\chl{\text{\sf c.h.}}

\def\ij{i,\:j}

\def\gl{\mathfrak{gl}}
\def\gln{\mathfrak{gl}_N}

\def\Uen{\Ue(\Tilde{\gln}\<)}

\def\beq{\begin{equation}}
\def\eeq{\end{equation}}
\def\be{\begin{equation*}}
\def\ee{\end{equation*}}

\nc{\bea}{\begin{eqnarray*}}
\nc{\eea}{\end{eqnarray*}}
\nc{\bean}{\begin{eqnarray}}
\nc{\eean}{\end{eqnarray}}
\nc{\Ref}[1]{{\rm(\ref{#1})}}

\def\Eh{\Hat E}

\def\Phh{\Hat P}

\def\Qh{\Hat Q}

\def\Wh{\rlap{$\>\Hat{\?\phantom W\<}\:$}W}

\def\psih{\Hat\psi}

\def\Jt{\tilde J}

\def\Ut{\Tilde U}

\def\Wt{\rlap{$\>\Tilde{\?\phantom W\<}\:$}W}

\def\N{\mathbb{Z}_{\geq0}}
\def\R{{\mathbb R}}

\def\XX{{\mc X}}

\def\II{{\mc I}}

\def\GG{{\bs\Gm}}

\def\TT{{\bs t}}
\def\ttt{{\bs t}}
\def\xx{{\bs x}}
\def\zz{{\bs z}}

\def\tti{\tilde{\<\TT\:}}

\nc{\bla}{{\bs\la}}
\nc{\Il}{{\II_{\bla}}}
\nc{\Fla}{\F_\bla}
\nc{\tfl}{{T^*\!\Fla}}
\nc{\GL}{{GL_n(\C)}}
\nc{\GLC}{{GL_n(\C)\times\C^*}}

\def\mub{{\bs\mu}}
\def\nub{{\bs\nu}}

\def\Czh{\C[\zz^{\pm1}\?,h^{\pm1}\:]}

\def\CZS{\C[\zz^{\pm1}]^{\>S_n}}

\def\CGs{\C[\:\GG^{\pm1}\:]^{\>S_\bla}}

\def\zzz{z_1\lc z_n}

\def\Imx{I^{\:\max}}

\def\zzzsi{z_{\si(1)}\lc z_{\si(n)}}

\def\Bck{\B^{\:q}}
\def\Bckp{\B^{\:'\?q'}}

\def\ddk_#1{q_{#1}\:\frac\der{\der\:q_{#1}}}

\def\bul{\mathbin{\raise.2ex\hbox{$\sssize\bullet$}}}
\def\intt{\mathchoice
{\mathop{\raise.2ex\rlap{$\,\,\ssize\backslash$}{\intop}}\nolimits}
{\mathop{\raise.3ex\rlap{$\,\sssize\backslash$}{\intop}}\nolimits}
{\mathop{\raise.1ex\rlap{$\sssize\>\backslash$}{\intop}}\nolimits}
{\mathop{\rlap{$\sssize\:\backslash$}{\intop}}\nolimits}}

\def\GZ/{Gelfand\:-Zetlin}
\def\GZi/{Gelfand\:-\<Zetlin}
\def\KZ/{{\slshape KZ\/}}
\def\qKZ/{{\slshape qKZ\/}}
\def\XXX/{{\slshape XXX\/}}
\def\XXZ/{{\slshape XXZ\/}}

\def\Sym{\on{Sym}}
\def\St{{\on{Stab}}}
\def\Stab{\on{Stab}}

\def\Slope{\on{Slope}}
\def\Loc{\on{Loc}}

\DeclareMathOperator\pt{pt}
\DeclareMathOperator\Pic{Pic}

\let\bW\Wbar
\let\hE\Eh

\DeclareMathOperator\codim{codim}

\def\%#1{^{\bra#1\ket}}

\def\WT{\rlap{$\>\Hat{\?\phantom W\<}\:$}W} % my

\begin{document}

\hrule width0pt
\vsk->

\title[Stable envelopes and Newton polytopes]
{Elliptic and $\bs K\<$-theoretic stable envelopes\\[3pt] and Newton polytopes}

\author
[R.\,Rim\'anyi, V\<.\,Tarasov, A.\:Varchenko]
{ R.\,Rim\'anyi$\>^{\star}$,
V\<.\,Tarasov$\>^\circ$, A.\:Varchenko$\>^\diamond$}

\maketitle

\begin{center}
{\it $^{\star\,\diamond}\<$Department of Mathematics, University
of North Carolina at Chapel Hill\\ Chapel Hill, NC 27599-3250, USA\/}

\vsk.5>
{\it $\kern-.4em^\circ\<$Department of Mathematical Sciences,
Indiana University\,--\>Purdue University Indianapolis\kern-.4em\\
402 North Blackford St, Indianapolis, IN 46202-3216, USA\/}

\vsk.5>
{\it $^\circ\<$St.\,Petersburg Branch of Steklov Mathematical Institute\\
Fontanka 27, St.\,Petersburg, 191023, Russia\/}
\end{center}

{\let\thefootnote\relax
\footnotetext{\vsk-.8>\noindent
$^\star\<${\sl E\>-mail}:\enspace rimanyi@email.unc.edu\>,
supported in part by NSF grant DMS-1200685\\
$^\circ\<${\sl E\>-mail}:\enspace vtarasov@iupui.edu\>, vt@pdmi.ras.ru\>,
supported in part by Simons Foundation grant 430235\\
$^\diamond\<${\sl E\>-mail}:\enspace anv@email.unc.edu\>,
supported in part by NSF grants DMS-1362924, DMS-1665239}}

\begin{abstract}
In this paper we consider the cotangent bundles of partial flag varieties.
We construct the $K\<$-theoretic stable envelopes for them and also define
a version of the elliptic stable envelopes. We expect that our elliptic stable
envelopes coincide with the elliptic stable envelopes defined by M.\,Aganagic
and A.\,Okounkov. We give formulas for the $K\<$-theoretic stable envelopes
and our elliptic
stable envelopes. We show that the $K\<$-theoretic stable envelopes are
suitable limits of our elliptic stable envelopes. That phenomenon was predicted
by M.\,Aganagic and A.\,Okounkov. Our stable envelopes are constructed in terms
of the elliptic and trigonometric weight functions which originally appeared in
the theory of integral representations of solutions of \qKZ/ equations twenty
years ago. (More precisely, the elliptic weight functions had appeared earlier
only for the $\frak{gl}_2$ case.) We prove new properties of the trigonometric
weight functions. Namely, we consider certain evaluations of the trigonometric
weight functions, which are multivariable Laurent polynomials, and show that
the Newton polytopes of the evaluations are embedded in the Newton polytopes of
the corresponding diagonal evaluations. That property implies the fact that the
trigonometric weight functions project to the $K\<$-theoretic stable envelopes.
\end{abstract}

\setcounter{footnote}{0}
\renewcommand{\thefootnote}{\arabic{footnote}}

{\small \tableofcontents }

\section{Introduction}
Malik and Okounkov started in \cite{MO1} a program to relate the quantum torus
equivariant generalized cohomology (cohomology, $K\<$-theory, elliptic
cohomology) of Nakajima varieties and representation theory of quantum groups.
A central role in that program is played by the stable envelopes, which are
maps from the equivariant generalized cohomology of the fixed point set of
the torus action to the equivariant generalized cohomology of the variety.
Stable envelopes depend on the choice of a chamber (a connected component of
the complement of an arrangement of real hyperplanes) and stable envelopes of
different chambers are related by $R$-matrices of the corresponding quantum
group.

The basic example of a Nakajima variety is the cotangent bundle of a variety of partial flags in $\C^n$. Such
a variety $\F_\bla$ is labeled by
$\bla=(\la_1,\dots,\la_N)\in\Z_{\geq 0}^N$, $|\bla|=\la_1+\dots+\la_N=n$ and consists of flags
\be
F_1\subset F_2\subset\dots\subset F_N=\C^n,
\ee
with $\dim F_i/F_{i-1}=\la_i$, $i=1,\dots,N$. The torus is
$T=(\C^\times)^{n}\times \C^\times$, with $(\C^\times)^n$ acting by diagonal matrices on
$\C^n$ and $\C^\times$ acting by multiplication on the cotangent
spaces. In this case the stable envelope map can be thought of as a linear map $\Stab_\si$ from
$(\C^N)^{\ox n}$ to the $T$-equivariant generalized cohomology of the disjoint union
$\sqcup_{\bla\in\Z_{\geq 0}^N,\, |\bla|=n }\tfl$ of cotangent bundles. The stable envelope map
depends on an element $\si$ of the symmetric group $S_n$.

Let $v_1,\dots,v_N$ be the standard basis of $\C^N$. The standard basis of $(\C^N)^{\ox n}$ consists of vectors
$v_I$ labeled by partitions $I=(I_1,\dots,I_N)$ of $\{1,\dots, n\}$,
$v_I=v_{i_1}\otimes \dots\ox v_{i_n}$,
where $i_a=j$ if $a\in I_j$. Let $\Il$ be the set of all partitions $I$ such that $|I_j|=\la_j$, $j=1,\dots,N$.
Then $\Stab_\si$ is the same as a collection of cohomology classes $(\ka_{\si\<,\>I})_{I\in\Il}$ of $\tfl$
for all $\bla\in\Z_{\geq 0}^N,\, |\bla|=n$. The classes must satisfy some remarkable defining relations.

The existence and uniqueness of the stable envelopes for the $T$-equivariant cohomology of Nakajima varieties
was established in \cite{MO1}.
The special case of the partial flag varieties see in \cite{GRTV,RTV1}, where the formulas for the stable envelopes were given.
The $K\<$-theoretic stable envelopes were defined by Maulik and Okounkov in the paper \cite{MO2}, which is in preparation.
According to \cite{O, OS}, the paper \cite{MO2} will also contain the theorem of existence and uniqueness of
the $K\<$-theoretic stable envelopes for Nakajima varieties.
One of the main results of this paper is the proof of the existence of the $K\<$-theoretic stable envelopes for cotangent bundles of
partial flag varieties and formulas for them, see Theorem \ref{thm:OisW}.

The definition of elliptic stable envelopes for Nakajima varieties was sketched
by Aganagic and Okounkov in \cite{AO}. In \cite{FRV} a version of the elliptic
stable envelopes was defined for the cotangent bundles of Grassmannians. The
second main result of this paper is the definition of a version of the elliptic
stable envelopes for cotangent bundles of partial flag varieties in the spirit
of \cite{FRV} and an axiomatic description of them in Theorem \ref{thm:ax}.
That definition coincides with the definition in \cite{FRV} for the cotangent
bundles of Grassmannians. We expect that the stable envelopes in \cite{FRV}
and this paper coincide with the corresponding elliptic stable envelopes
in \cite{AO}.

Our formulas for different versions of stable envelopes are given in terms of
the {\it weight functions}, which originally appeared in the theory of integral
representations for solutions of different versions of \qKZ/ equations and associated Bethe ansatz, see
\cite{SV,V,TV1,TV2,TV3,TV4,FTV1,FTV2, MTV}. The integral representation for a solution
$I(z_1,\dots,z_n,y)$ of the Yangian $Y_y(\gl_N)$ \qKZ/ equation
with values in $(\C^N)^{\ox n}$ has the form
\be
I(z_1,\dots,z_n,y) = \sum_I \left(\int\Phi(\ttt, z_1,\dots,z_n,y) W_I(\ttt,z_1,\dots,z_n,y)d\ttt \right) v_I,
\ee
where $\Phi(\ttt, \zz,y)$ is some scalar {\it master function} and $W_I(\ttt,\zz,y)$ are the weight functions,
$y$ is deformation parameter of the Yangian, $\ttt=(t^{(k)}_i)$ are the integration variables. To obtain a $T$-equivariant
cohomological stable envelope we identify the variables $z_1,\dots,z_n,y$ with the equivariant parameters
of the torus $T=(\C^\times)^{n}\times \C^\times$ and
the variables $\ttt$ with equivariant Chern roots of the tautological vector bundles on $\tfl$.
To construct the $K\<$-theoretic or elliptic stable envelopes we take the
weight functions appearing in other versions of the \qKZ/ equations.

Initially the weight functions were invented to construct the integral
representations for solutions of the \qKZ/ equations. The main requirement
on them was the condition that the weight functions
$(W_{\si\<,\>I}(\ttt,\zz,y))_{I\in\Il}$ and
$(W_{\si',I}(\ttt,\zz,y))_{I\in\Il}$ labeled by $\si,\si'\in S_n$ must be
related by the $R$-matrices of the corresponding quantum group, which permute
the factors in the tensor product of $n$ evaluation representations.
No stable envelope properties were expected from them. Only now, 20 years
later, when the stable envelopes were introduced, the new stable envelope
properties of the weight functions are becoming uncovered.

In this paper we study the $\frak{gl}_N$ elliptic and trigonometric weight
functions. For $N=2$, the elliptic weight functions were introduced in
\cite{FTV1, FTV2}. For $N>2$, the $\frak{gl}_N$ elliptic weight functions
did not appear previously in the literature. The $\frak{gl}_N$ trigonometric
weight functions of this paper are modifications of the trigonometric weight
functions considered in \cite{TV1, TV4}.

In this paper we consider, in particular, the
trigonometric weight functions denoted by $(\Wt^\Dl_{\si\<,\>I}(\ttt,\zz,h))_{I\in\Il}$ and their evaluations
$(\Wt^\Dl_{\si\<,\>I}(\zz_J,\zz,h))_{I\in\Il}$ labeled
by partitions $J\in\Il$. Each of these evaluations is a Laurent polynomial in $z_1,\dots,z_n$. We show that
the diagonal evaluation $\Wt^\Dl_{\si\<,\>J}(\zz_J,\zz,h)$ equals some explicit product of binomials and the Newton polytope
of any off-diagonal evaluation $\Wt^\Dl_{\si\<,\>I}(\zz_J,\zz,h)$ with $I\ne J$ can be parallelly moved inside the Newton polytope
of the corresponding diagonal evaluation $\Wt^\Dl_{\si\<,\>J}(\zz_J,\zz,h)$, see Theorem \ref{thm:Newton}. This is the stable envelope property of
the trigonometric weight functions. Theorem \ref{thm:Newton} is our third main result.

We emphasize that the weight functions are not the same object as the stable
envelopes. The weight functions are polynomials of certain variables, while
the stable envelopes are projections of the weight functions (usually divided
by some nontrivial factors) to the corresponding equivariant cohomology
algebras with relations.

Our fourth main result is the relation between our elliptic stable envelopes,
introduced in this paper, and the trigonometric stable envelopes, see
Theorem~\ref{pr4.1}. That relation between the elliptic and $K\<$-theoretic
stable envelopes was predicted in \cite[Proposition 3.5]{AO}, see a remark
at the end of Section~\ref{qto0}.
%after Theorem~\ref{pr4.1}.
The elliptic weight functions are defined as symmetrizations of alternating
products of theta functions of one variable. When the modular parameter and the
argument of the theta function tend to zero in a special way the theta function
turns into a binomial. The corresponding limit of an elliptic weight function
turns into a trigonometric weight function. In that way the trigonometric
stable envelopes can be recovered from the elliptic stable envelopes.

The exposition is the following. In Section \ref{sec:2} we introduce the elliptic weight functions and study their properties.
Theorem \ref{thm:recur} shows that different elliptic weight functions are related by the $\gl_N$ elliptic dynamical $R$-matrices.
Theorem \ref{thm orth} gives the orthogonality relations for
the elliptic weight functions, similarly to the orthogonality relations in
\cite{TV2, TV3, RTV1, RTV2, FRV}. In Section \ref{sec:trigweight} we introduce
the trigonometric weight functions. Theorem \ref{thm:r_tri} describes
the $R$-matrix properties of the trigonometric weight functions.
In Section \ref{sec:trigweight} we formulate and prove Theorem~\ref{thm:Newton}
on Newton polytopes.

In Section \ref{Preliminaries from Geometry} we discuss elementary facts on partial flag varieties.
In Section \ref{sec:5} the $K\<$-theoretic stable envelopes of cotangent bundles of partial flag varieties are constructed.
In Section \ref{sec:6} we remind elementary facts on line bundles on powers of an elliptic curve.
In Section \ref{sec7} we define our elliptic stable envelopes. At the end of the paper we
comment on the definitions of stable envelopes in cohomology, $K\<$-theory, and elliptic cohomology.

This paper was inspired by papers and oral presentations by Andrei Okounkov
and his coauthors. Our goal was to understand the $K\<$-theoretic and elliptic
stable envelopes and their relations with weight functions. The authors thank
G.\,Felder for useful discussions and the referee for helpful suggestions.
The authors thank H.\,Konno who pointed out a mistake in the earlier version
of Theorem~\ref{thm orth}, in which the shift \,$h^\bla$ had been missing,
see that shift in our earlier papers \cite{RV1, FRV}. The third author thanks
the Max Planck Institute for Mathematics and Hausdorff Institute for
Mathematics in Bonn for hospitality.

\noindent{\em Remark on notation.}
Weight functions appear in the literature in three flavors: rational,
trigonometric, and elliptic. Their natural notation would be
$W_{\si\<,\>I}^{rat}, W^{trig}_{\si,I,\Dl}, W^{ell}_{\si\<,\>I}$.
We will use the latter for the elliptic version. For the trigonometric one,
we will write $W^\Dl_{\si\<,\>I}$ to keep the notation simpler.
There will be no rational weight functions in this paper.

\section{Elliptic weight functions}
\label{sec:2}

\subsection{Notation}
Let $\tau$ be a complex number with positive imaginary part and let
$q=e^{2\pi i \tau}$. We will use the complex variables $x$ and $u$ satisfying
$x=e^{2\pi i u}$. We set $q^{1/2}=e^{\pi i \tau}$ and $x^{1/2}=e^{\pi i u}$.
Define the theta function
\beq
\label{sqr}
\thi(x)\,=\,(x^{1/2}\?-x^{-1/2})\,\pho(qx)\>\pho(q/x)\,,\qquad
\pho(x)\,=\,\prod_{s=0}^\infty\,(1-q^s x)\,,
\vv-.3>
\eeq
cf.~\cite[(67)\:]{AO}\>. Then
\vvn-.2>
\beq
\label{one}
\frac{\thi(qx)}{\thi(x)}\,=\,-\>\frac1{q^{1/2}x}\;,\qquad
\thi(1/x)\,=\>-\,\thi(x)\,,
\vv.3>
\eeq
Let
\beq
\label{TH}
\theta(u)\,=\,\theta(u,\tau)\,=\,\thi(e^{2\pi iu}, e^{2\pi i\tau})
\eeq
be obtained from $\thi(x,q)$ by the substitution
$\,q=e^{2\pi i\tau},\;\;x=e^{2\pi iu}$\,. Then
\vvn.3>
\beq
\label{u+tau}
\theta(u+1)\,=\>-\,\theta(u)\,,\qquad
\theta(u+\tau)\,=\>-\,e^{-\pi i \tau-2\pi i u}\>\theta(u)\,.
\eeq

\vsk.3>
Let \,$q\to 0$\,, $x,a,b$ \,are fixed, \,$0<\Re\>\eps<1$\,, and \,$m\in\Z$\>,
then
\vvn.2>
\beq \label{eqn:thetl}
\thi(x)\to x^{1/2}\?-x^{-1/2}\:,\qquad
\frac{\thi(a\:q^{\>m+\eps})}{\thi(b\:q^{\>m+\eps})}\,\to\,(a/b)^{-m\:-1/2}\:.
\eeq

\vsk.3>
The following identities holds for \,$\thi(x)$:
\vvn.4>
\begin{align}
\label{two}
& \thi(\al\:y_1/x)\,\thi(h\:y_2/x)\,\thi(h\:y_1/y_2)\,\thi(\al)\,={}
\\[3pt]
&\!\<{}=\,\thi(\al\:h\:y_1/x)\,\thi(y_2/x)\,\thi(y_1/y_2)\,\thi(\al/h)
\>+\>\thi(h\:y_1/x)\,\thi(\al\:y_2/x)\,\thi(h)\,\thi(\al\:y_1/y_2)\,,
\notag
\end{align}
\begin{align}
\label{three}
& \thi(\al_1\al_2\:h\:y_1/x_1)\,\thi(y_2/x_1)\,\thi(h\:y_1/x_2)\,
\thi(\al_2\:h\:y_2/x_2)\,\thi(h\:x_2/x_1)\,\thi(y_1/y_2)\,\thi(\al_1)\>-{}
\\[3pt]
&{}\;\;\;-\,\thi(h\:y_1/x_1)\,\thi(\al_2\:h\:y_2/x_1)\,
\thi(\al_1\al_2\:h\:y_1/x_2)\,\thi(y_2/x_2)\,\thi(\al_1x_2/x_1)\,
\thi(y_1/y_2)\,\thi(h)\,={}
\notag
\\[3pt]
&\!{}=\,\thi(\al_1\al_2\:h\:y_1/x_1)\,\thi(h\:y_2/x_1)\,\thi(y_1/x_2)\,
\thi(\al_2\:h\:y_2/x_2)\,\thi(x_2/x_1)\,\thi(h\:y_1/y_2)\,\thi(\al_1)\>-{}
\notag
\\[3pt]
&{}\;\;\;-\,\thi(h\:y_1/x_1)\,\thi(\al_1\al_2\:h\:y_2/x_1)\,
\thi(\al_2\:h\:y_1/x_2)\,\thi(y_2/x_2)\,\thi(x_2/x_1)\,\thi(\al_1y_1/y_2)\,
\thi(h)\,.
\notag
\end{align}

\subsection{Elliptic \,$R$-matrix}

Given a positive integer \,$N$, let \,$\mub=(\mu_1\lc\mu_N)$\,.
Nonzero entries of Felder's dynamical \,$R$-matrix \,$^F\!\Rc(x,\mub)$ \,are
\vvn.2>
\be
^F\!\Rc_{jj}^{jj}(x,\mub)\,=\,1\,,\quad\;
^F\!\Rc_{jk}^{jk}(x,\mub)\,=\,
\frac{\thi(x)\,\thi(h\mu_j/\mu_k)}{\thi(xh)\,\thi(\mu_j/\mu_k)}\;,\quad\;
^F\!\Rc_{kj}^{jk}(x,\mub)\,=\,
\frac{\thi(x\mu_j/\mu_k)\,\thi(h)}{\thi(xh)\,\thi(\mu_j/\mu_k)}\;,\kern-1.6em
\vv.3>
\ee
where $j\ne k$\>, see~\cite{F}\,, where \,$h=e^{-2\pi i\gm}$ and
\;$\mu_j=e^{-2\pi i\la_j}$\,, \,$j=1\lc N$\:.
The \,$R$-matrix \,$\Rc(x,\mub)$
\vvn-.6>
\beq
\label{Rm}
\Rc_{jj}^{jj}(x,\mub)\,=\,1\,,\qquad\;
\Rc_{kj}^{jk}(x,\mub)\,=\,
\frac{\thi(x\mu_j/\mu_k)\,\thi(h)}{\thi(xh)\,\thi(\mu_j/\mu_k)}\;,\qquad
j\ne k\,,\kern-1em
\eeq
\be
\Rc_{jk}^{jk}(x,\mub)\,=\,\frac{\thi(x)\,\thi(h\mu_j/\mu_k)\,\thi(h\mu_k/\mu_j)}
{\thi(xh)\,\thi(\mu_j/\mu_k)\,\thi(\mu_k/\mu_j)}\;,\qquad
\Rc_{kj}^{kj}(x,\mub)\,=\,\frac{\thi(x)}{\thi(xh)}\;,\qquad j<k\,,\kern-2em
\vv.6>
\ee
is similar (in the case \,$N=2$\>) to the $R$-matrix in
\cite[formula (66)\:]{AO} under the identification
\,$x=u\,,\;\;h=1/\hbar\,,\;\;\mu_1/\mu_2\<=z$\,.

\subsection{Index set $\Il$, variables $\ttt$}
\label{sec:ilt}

We will use the following notations throughout the paper.
Let $n,N\in \N$ and let $\bla\in \N^N$ be such that $\sum_{k=1}^N \la_k=n$.
Set \,$\la^{(k)}\<=\la_1\lsym+\la_k$ \,for $k=0\lc N$\>,
and $\la^{\{1\}}\<=\la^{(1)}\lsym+\la^{(N-1)}$\>.

The set of partitions \,$I=(I_1\lc I_N)$ of \,$\{1\lc n\}$ \,with
\,$|I_k|=\la_k$ is denoted by \,$\Il$. For $I\?\in \Il$ we will use the
notation $\bigcup_{a=1}^k I_a = \{i^{(k)}_1<\ldots< i^{(k)}_{\la^{(k)}}\}.$

Consider variables $t^{(k)}_a$ for $k=1\lc N$, $a=1\lc\la^{(k)}$,
where $t^{(N)}_a\?=z_a$\,, \,$a=1\lc n$. Denote
$t^{(j)}=(t^{(j)}_k)_{k\leq\la^{(j)}}$ and \,$\ttt=(t^{(1)}\lc t^{(N-1)})$.

\subsection{Elliptic weight functions}
\label{sec:ellwe}

For $I\?\in\Il$ define the {\it elliptic weight function}
\beq
\label{WI}
W^{\>\on{ell}}_I(\ttt,\zz,h,\mub)\,=\,\bigl(\thi(h)\bigr)^{\la^{\{1\}}}\<
\Sym_{\>t^{(1)}} \ldots \Sym_{\>t^{(N-1)}} U_I(\ttt,\zz,h,\mub)\,,
\eeq
where \;$\Sym_{\>t^{(k)}}$ \,is the symmetrization with respect to
the variables \,$t^{(k)}_1\lc t^{(k)}_{\la^{(k)}}$,
\vvn.1>
\be
\Sym_{\>t^{(k)}} f\bigl(t^{(k)}_1\lc t^{(k)}_{\la^{(k)}}\bigr)\,=\!
\sum_{\si\in S_{\la^{(k)}}\!\!}\<
f\bigl(t^{(k)}_{\si(1)}\lc t^{(k)}_{\si(\la^{(k)})}\bigr)\,,
\vv-.2>
\ee
\beq
\label{UI}
U^{\>\on{ell}}_I(\ttt,\zz,h,\mub)\,=\,\prod_{k=1}^{N-1}\,\prod_{a=1}^{\la^{(k)}}\,\biggl(\,
\prod_{c=1}^{\la^{(k+1)}}\,\psi^{\>\on{ell}}_{I,k,a,c}(t^{(k+1)}_c/t^{(k)}_a)\,
\prod_{b=a+1}^{\la^{(k)}}\?\frac{\thi(ht^{(k)}_b/t^{(k)}_a)}
{\thi(t^{(k)}_b/t^{(k)}_a)}\,\biggr)\>,
\vv.5>
\eeq
\beq
\label{psi}
\psi^{\>\on{ell}}_{I,k,a,c}(x)\,=\left\{\begin{alignedat}3
&\kern4.56em\thi(hx)\,,&&\text{if} && i^{(k+1)}_c\!<i^{(k)}_a\,,\\[2pt]
&\,{\dsize\frac
{\thi(x\:h^{1+p_{I\<\<,\:j(I\<,\>k,\:a)}(i_a^{(k)})-\:p_{I\<\<,\:k+1}(i_a^{(k)})}
\mu_{k+1}/\mu_{j(I\<,\>k,\:a)})}
{\thi(h^{1+p_{I\<\<,\:j(I\<,\>k,\:a)}(i_a^{(k)})-\:p_{I\<\<,\:k+1}(i_a^{(k)})}
\mu_{k+1}/\mu_{j(I\<,\>k,\:a)})}}\;,\;\quad &&
\text{if}\;\quad&& i^{(k+1)}_c\!=i^{(k)}_a\,,
\\[2pt]
&\kern4.56em\,\thi(x)\,,&&\text{if} && i^{(k+1)}_c\!>i^{(k)}_a\,,
\end{alignedat}\right.\kern-1.6em
\eeq
where \,${j(I\<,\>k,\:a)}\in\{1,\dots,N\}$ \,is such that
\,$i^{(k)}_a\<\in\<I_{j(I\<,\>k,\:a)}$\,, \,and
\vvn.3>
\beq
\label{pjk}
p_{I\<\<,\:j}(m)\,=\,|\,I_j\cap\{1\lc m-1\}\>|\,,\qquad j=1\lc N\>.\kern-2em
\eeq

\vsk.3>
Denote by $\psi_I(h,\mub)$ the product of all denominators appearing
in \Ref{UI} and depending on $h,\mub$ \,only,
\vvn-.5>
\beq
\label{psI}
\psi_I(h,\mub)\,=\,\prod_{k=1}^{N-1}\,\prod_{k=1}^{\la^{(k)}}\,
\thi(h^{1+p_{I\<\<,\:j(I\<,\>k,\:a)}
(i_a^{(k)})-\:p_{I\<\<,\:k+1}(i_a^{(k)})}\mu_{k+1}/\mu_{j(I\<,\>k,\:a)})\,.
\eeq

\begin{rem}
Notice that the weight functions are regular at the diagonals $t^{(k)}_b=t^{(k)}_a$
despite the appearance of the denominators $\thi(t^{(k)}_b/t^{(k)}_a)$ in \Ref{UI}.
\end{rem}

For \,$\si\in S_n$ \>and \,$I\?\in\Il$\>, define the elliptic weight function
\vvn.4>
\beq
\label{Wsi}
W^{\>\on{ell}}_{\si\<,\>I}(\ttt,\zz, h,\mub)\,=\,
W^{\>\on{ell}}_{\si^{-1}(I)}(\ttt,z_{\si(1)}\lc z_{\si(n)},h,\mub)\,,
\vv.2>
\eeq
where \,$\si^{-1}(I)=\bigl(\si^{-1}(I_1)\lc\si^{-1}(I_N)\bigr)$.
Hence, \,$W^{\>\on{ell}}_I=W^{\>\on{ell}}_{\id\:,\:I}$.

\begin{example}
For $N=2$, $n=2$, $\lambda=(1,1)$, $s=(1,2)\in S_2$, $t^{(1)}_1\!=t$\,,
we have
\begin{align*}
W^{\>\on{ell}}_{\id,(\{1\},\{2\})} ={} &\>
\thi(h) \thi(z_2/t) \frac{\thi(hz_1/t \cdot\mu_2/\mu_1)}{\thi(h\mu_2/\mu_1)}\;,
\\
W^{\>\on{ell}}_{\id,(\{2\},\{1\})} ={} &\>
\thi(h) \thi(hz_1/t) \frac{\thi(z_2/t \cdot\mu_2/\mu_1)}{\thi(\mu_2/\mu_1)}\;,
\\
W^{\>\on{ell}}_{s,(\{1\},\{2\})} ={} &\>
\thi(h) \thi(hz_2/t) \frac{\thi(z_1/t \cdot\mu_2/\mu_1)}{\thi(\mu_2/\mu_1)}\;,
\\
W^{\>\on{ell}}_{s,(\{2\},\{1\})} ={} &\>
\thi(h) \thi(z_1/t) \frac{\thi(hz_2/t \cdot\mu_2/\mu_1)}{\thi(h\mu_2/\mu_1)}\;.
\end{align*}
% where we used the shorthand notation $t=t^{(1)}_1$ for the only occurring
% $t$-variable.
\end{example}

\subsection{Exchange properties}
Let
\beq
\label{btx}
\bt(x_1,x_2,y_1,y_2,\al)\,=\,\Sym_{\>x_1,\>x_2}\thi(\al\:y_1/x_1)\,
\thi(y_2/x_1)\,\thi(h\:y_1/x_2)\,\thi(\al\:h\:y_2/x_2)\,
\frac{\thi(h\:x_2/x_1)}{\thi(x_2/x_1)}\;.\kern-2em
\eeq
\begin{lem}
\label{btxy}
$\,\beta(x_1,x_2,y_1,y_2,\al)$ \>is symmetric in \,$y_1,y_2$.
\end{lem}
\begin{proof}
Identity \Ref{three} for \,$\al_1=1/h\,,\;\al_2=\al$ \,is equivalent to
\beq
\bt(x_1,x_2,y_1,y_2,\al)\,=\,\bt(\al\:h\:y_1,\al\:h\:y_2,x_1,x_2,\al)\,,
\eeq
which proves the claim.
\end{proof}

Let \,$s_{\ij}\in S_n$ \>denote the transposition of \>$i$ \>and \>$j$\,.
Set \,\rlap{$\mub^{(i)}_I\<=(h^{-p_{I\<\<,1}(i)}\mu_1\lc
h^{-p_{I\<\<,N}(i)}\mu_N)$\,.}

\begin{thm}
\label{thm:recur}
Let \,$\si\in S_n$ \>be such that \,$\si(i)\in I_a$ and \,$\si(i+1)\in I_b$\>.
Then
\vv.3>
\beq
\label{a=b}
W^{\>\on{ell}}_{\si s_{i,i+1},\:I}\,=\,W^{\>\on{ell}}_{\si\<,\>I}
\vv.1>
\eeq
for \,$a=b$\,, and
\vvn.2>
\beq
\label{a<>b}
W^{\>\on{ell}}_{\si s_{i,i+1},\:I}\>=\,
\Rc_{ab}^{ab}(z_{\si(i)}/z_{\si(i+1)},\mub_{\si^{-1}(I)}^{(i)})
\,W^{\>\on{ell}}_{\si\<,\>I}\,+\>
\Rc_{ab}^{b\:a}(z_{\si(i)}/z_{\si(i+1)},\mub_{\si^{-1}(I)}^{(i)})
\,W^{\>\on{ell}}_{\si,\,s_{\si(i),\:\si(i+1)}(I)}\kern-2em
\vv.2>
\eeq
%\begin{align}
%\label{a<>b}
%W^{\>\on{ell}}_{\si s_{i,i+1},\:I}\>={}\, &
%\Rc_{ab}^{ab}(z_{\si(i)}/z_{\si(i+1)},\mub_{\si^{-1}(I)}^{(i)})
%\,W^{\>\on{ell}}_{\si\<,\>I}\,+{}
%\\[3pt]
%\notag
%{}=\,{} & \Rc_{ab}^{b\:a}(z_{\si(i)}/z_{\si(i+1)},\mub_{\si^{-1}(I)}^{(i)})
%\,W^{\>\on{ell}}_{\si,\,s_{\si(i),\:\si(i+1)}(I)}\kern-.6em
%\end{align}
for \,$a\ne b$\,.
\end{thm}
\begin{proof}
By formula \Ref{Wsi}\:, it suffices to prove the statement for \,$\si$
\,being the identity permutation. In that case, \,$i\in I_a$\,,
\,$i+1\in I_b$\,, formula \Ref{a=b} for \,$a=b$ \,reads
\vvn.3>
\beq
\label{abeq}
W^{\>\on{ell}}_I(\ttt,\zz^{(i)}\!,h,\mub)\,=\,
W^{\>\on{ell}}_I(\ttt,\zz,h,\mub)\,,
\vv.1>
\eeq
where \,$\zz^{(i)}=(z_1\lc z_{i-1},z_{i+1},z_i,z_{i+2}\lc z_n)$\,,
and formula \Ref{a<>b} for \,$a\ne b$ \,reads
\vvn.2>
\begin{align}
\label{abne}
W^{\>\on{ell}}_{s_{i,i+1}(I)}(\ttt,\zz^{(i)}\!,h,\mub)\,=\,{} &
\Rc_{ab}^{ab}(z_i/z_{i+1},\mub_I^{(i)})\,
W^{\>\on{ell}}_I(\ttt,\zz,h,\mub)\>+{}
\\[3pt]
\notag
{}+\,{} & \Rc_{ab}^{b\:a}(z_i/z_{i+1},\mub_I^{(i)})
\,W^{\>\on{ell}}_{s_{i,i+1}(I)}(\ttt,\zz,h,\mub)\,.\kern-2em
\end{align}

\goodbreak
\noindent
Furthermore, formula \Ref{UI} implies that it suffices to consider only
\vv.08>
the case \,$n=2$\>, \,$i=1$\>. By formula \Ref{UI}, the products
\,$U^{\>\on{ell}}_{s_{i,i+1}(I)}(\ttt,\zz^{(i)}\!,h,\mub)$\,,
\,$U^{\>\on{ell}}_I(\ttt,\zz,h,\mub)$\,, and
\vv.1>
\,$U^{\>\on{ell}}_{s_{i,i+1}(I)}(\ttt,\zz,h,\mub)$ \,have many common factors,
being different in the part that reproduces the case \,$n=2$\>, \,$i=1$
\,up to a change of notation. Then formula \Ref{abne} in general can be
obtained by taking formula \Ref{abne} for \,$n=2$\>, \,$i=1$ \,in appropriate
\vv.06>
variables, multiplying it by the common factors of the products
\,$U^{\>\on{ell}}_{s_{i,i+1}(I)}(\ttt,\zz^{(i)}\!,h,\mub)$\,,
\,$U^{\>\on{ell}}_I(\ttt,\zz,h,\mub)$\,,
\vv.08>
\,$U^{\>\on{ell}}_{s_{i,i+1}(I)}(\ttt,\zz,h,\mub)$\,, and subsequent
symmetrization in \,$\TT^{(j)}$ variables for each \,$j=1\lc N-1$\,.

\vsk.2>
Assume \,$n=2$\>, \,$i=1$\,. To simplify the notation, we write
\,$s=s_{1,2}$\,. The proof uses the recursive structure of the weight
functions. Set
\vvn.3>
\be
\Ut_I\%k(\TT^{(k)}\<,\:\TT^{(k+1)})\,=\,\prod_{a=1}^{\la^{(k)}}\,\biggl(\,
\prod_{c=1}^{\la^{(k+1)}}\,\psi^{\>\on{ell}}_{k,a,c}(t^{(k+1)}_c/t^{(k)}_a)\,
\prod_{b=a+1}^{\la^{(k)}}\?\frac{\thi(ht^{(k)}_b/t^{(k)}_a)}
{\thi(t^{(k)}_b/t^{(k)}_a)}\,\biggr)\>,
\vv.2>
\ee
so that \,$U^{\>\on{ell}}_I(\ttt,\zz)=
\prod_{k=1}^{N-1}\Ut_I\%k(\TT^{(k)}\<,\:\TT^{(k+1)})$\,,
see~\Ref{UI}\:, \Ref{psi}\:. Define
\vvn.4>
\beq
\label{WMI}
W_I\%k(\ttt\%{k-1}\<,\:\TT^{(k)})\,=\,
\Sym_{\>\TT^{(1)}} \ldots \Sym_{\>\TT^{(k-1)}}
\bigl(\>\Ut_I\%1(\TT^{(1)}\<,\:\TT^{(2)})\ldots
\Ut_I\%{k-1}(\TT^{(k-1)}\<,\:\TT^{(k)})\bigr)
\eeq
where \;$\TT\%j=(\TT^{(1)}\<\lc\:\TT^{(j)})$\,. In particular,
\beq
\label{WWN}
W^{\>\on{ell}}_I(\ttt,\zz)\,=\,
\bigl(\thi(h)\bigr)^{\la^{\{1\}}}W_I\%N(\ttt,\zz)\,.
\eeq
Formula~\Ref{WMI} yields
\vvn.1>
\beq
\label{WWU}
W_I\%{k+1}(\ttt\%k\<,\:\TT^{(k+1)})\,=\,
\Sym_{\>\TT^{(k)}}\bigl(\>W_I\%k(\ttt\%{k-1}\<,\;\TT^{(k)})\,
\Ut_I\%k(\TT^{(k)}\<,\:\TT^{(k+1)})\bigr)\,.
\eeq

\vsk.4>
Let \,$a=b$\,. Then \,$I_a\<=\{1\:,\<2\:\}$\,, \,$I_c\<=\Empty$ \,for
\,$c\ne a$\,,
\;$\Ut_I\%k\!=1$ \,for \,$k<a$\,, and
\vvn.3>
\be
\Ut_I\%k=\,\zt(t_1^{(k)}\?,\:t_2^{(k)}\?,\:t_1^{(k+1)}\?,\:t_2^{(k+1)}\?,
1/h\:,\:h\:\mu_{k+1}/\mu_a)
\vv.1>
\ee
for \,$k\ge a$\,, where
\vvn-.2>
\beq
\label{zt}
\zt(x_1,x_2,y_1,y_2,\al_1,\al_2)\,=\,
\thi(\al_1\al_2\:h\:y_1/x_1)\,\thi(y_2/x_1)\,\thi(h\:y_1/x_2)\,
\thi(\al_2\:h\:y_2/x_2)\,\frac{\thi(h\:x_2/x_1)}{\thi(x_2/x_1)}\;,\kern-2em
\eeq
cf.~\Ref{btx}\:. Therefore, by formula~\Ref{WWU} and Lemma~\ref{btxy}\:,
\vvn.3>
\be
W_I\%N(\ttt,z_1,z_2)\,=\,\prod_{k=a}^{N-1}\:
\bt(t_1^{(k)}\?,\:t_2^{(k)}\?,\:t_1^{(k+1)}\?,\:t_2^{(k+1)}\?,
\:h\:\mu_{k+1}/\mu_1)\,=\,W_I\%N(\ttt,z_2,z_1)\,.
\vv.3>
\ee
By formulae~\Ref{Wsi} and \Ref{WWN}\:, we have
\,$W^{\>\on{ell}}_{s,I}(\ttt,z_1,z_2)=W^{\>\on{ell}}_I(\ttt,z_2,z_1)=W^{\>\on{ell}}_I(\ttt,z_1,z_2)$\>,
which proves equality~\Ref{abeq}\:.

\vsk.3>
Let \,$a\ne b$\,. Then \,$I_a\<=\{1\:\}\,,$ \,$I_b=\{2\:\}$\,, and
\,$I_c\<=\Empty$ \,for \,$c\ne a\:,b$\,. To simplify writing, we will consider
only the case \,$a>b$\,. The proof in the case \,$a<b$ is completely similar.

\vsk.2>
Let \,$a>b$\,. Then \;$\Ut_I\%k\!=\Ut_{s(I)}\%k\<=1$ \,for \,$k<b$\,, and
\vvn-.2>
\begin{alignat}2
& \Ut_I\%k=\,\Ut_{s(I)}\%k\>=\,
\frac{\thi\bigl(h\:\mu_{k+1}\:t^{(k+1)}_1\!/(\mu_b\>t^{(k)}_1)\bigr)}
{\thi(h\:\mu_{k+1}/\mu_b)}\;, && b\le k<a-1\,,\kern-3em
\notag
\\[7pt]
& \Ut_I\%{a-1}=\,\frac{\thi(h\:t^{(a)}_1\!/t^{(a-1)}_1)\,
\thi\bigl(\mu_a\:t^{(a)}_2\!/(\mu_b\>t^{(a-1)}_1)\bigr)}
{\thi(\mu_a/\mu_b)}\;,
\notag
\\[5pt]
& \Ut_{s(I)}\%{a-1}=\,\frac{\thi(t^{(a)}_2\!/t^{(a-1)}_1)\,
\thi\bigl(h\:\mu_a\:t^{(a)}_1\!/(\mu_b\>t^{(a-1)}_1)\bigr)}
{\thi(h\:\mu_a/\mu_b)}\;,
\notag
\\[12pt]
\label{Utk}
& \Ut_I\%k=\,\zt(t_1^{(k)}\?,\:t_2^{(k)}\?,\:t_1^{(k+1)}\?,\:t_2^{(k+1)}\?,
\:\mu_b\:/\mu_a\:,\>\mu_{k+1}/\mu_b)\;, && k\ge a\,,
\\[6pt]
\label{Utsk}
& \Ut_{s(I)}\%k\>=\,
\zt(t_1^{(k)}\?,\:t_2^{(k)}\?,\:t_1^{(k+1)}\?,\:t_2^{(k+1)}\?,
\:\mu_a\:/\mu_b\:,\>\mu_{k+1}/\mu_a)\;,\qquad && k\ge a\,,
\\[-12pt]
\notag
\end{alignat}
where \,$\zt$ is defined by~\Ref{zt}\:. Set
\vvn.2>
\beq
\label{rho}
\rho(x)\,=\,\frac{\thi(xh)}{\thi(x)}\;,\qquad
\phi(x)\,=\,\frac{\thi(x\mu_b\:/\mu_a)\,\thi(h)}
{\thi(x)\,\thi(\mu_b\:/\mu_a)}\;.
\vv.3>
\eeq
We will prove the equality
\vvn.4>
\begin{align}
\label{RW}
& W_I\%k(\ttt\%{k-1}\<,\:t^{(k)}_1\<,\:t^{(k)}_2)\,={}
\\[5pt]
& \!{}=\,\rho(t^{(k)}_1\!/\:t^{(k)}_2)\,
W_{s(I)}\%k(\ttt\%{k-1}\<,\:t^{(k)}_2\<,\:t^{(k)}_1)\:-\:
\phi(t^{(k)}_1\!/\:t^{(k)}_2)\,
W_{s(I)}\%k(\ttt\%{k-1}\<,\:t^{(k)}_1\<,\:t^{(k)}_2)
\notag
\\[-14pt]
\notag
\end{align}
for \,$k\ge a$\,. By formulae~\Ref{Rm}\:, \Ref{Wsi}\:, \Ref{WWN}\:,
\Ref{rho}\:, this equality proves relation~\Ref{abne}\:.

\vsk.3>
The proof of formula~\Ref{RW} is by induction on \,$k$\,. The base of induction
is \,$k=a$\,, when formula~\Ref{RW} reduces to identity~\Ref{two}\:.
For the induction step, we multiply formula~\Ref{RW} by \,$\Ut_I\%k$,
see~\Ref{Utk}\:, and symmetrize over \,$t^{(k)}_1\<,\:t^{(k)}_2$:
\vvn.3>
\begin{align*}
& \Sym_{t^{(k)}_1\!,\>t^{(k)}_2}
W_I\%k(\ttt\%{k-1}\<,\:t^{(k)}_1\<,\:t^{(k)}_2)\;
\zt(t_1^{(k)}\?,\:t_2^{(k)}\?,\:t_1^{(k+1)}\?,\:t_2^{(k+1)}\?,
\:\mu_b\:/\mu_a\:,\>\mu_{k+1}/\mu_b)\,={}
\\[5pt]
& {}=\,\Sym_{t^{(k)}_1\!,\>t^{(k)}_2}\rho(t^{(k)}_1\!/\:t^{(k)}_2)\,
W_{s(I)}\%k(\ttt\%{k-1}\<,\:t^{(k)}_2\<,\:t^{(k)}_1)\;
\zt(t_1^{(k)}\?,\:t_2^{(k)}\?,\:t_1^{(k+1)}\?,\:t_2^{(k+1)}\?,
\:\mu_b\:/\mu_a\:,\>\mu_{k+1}/\mu_b)\,-{}
\\[5pt]
&\hp{{}={}}\!-\,\Sym_{t^{(k)}_1\!,\>t^{(k)}_2}\phi(t^{(k)}_1\!/\:t^{(k)}_2)
\,W_{s(I)}\%k(\ttt\%{k-1}\<,\:t^{(k)}_1\<,\:t^{(k)}_2)\;
\zt(t_1^{(k)}\?,\:t_2^{(k)}\?,\:t_1^{(k+1)}\?,\:t_2^{(k+1)}\?,
\:\mu_b\:/\mu_a\:,\>\mu_{k+1}/\mu_b)\,.
\\[-14pt]
\end{align*}
Now we evaluate the left\:-hand side by formula~\Ref{WWU} and swap
\,$t^{(k)}_1\<,\:t^{(k)}_2$ in the first summand in the right\:-hand side
to factor out the common factor
\,$W_{s(I)}\%k(\ttt\%{k-1}\<,\:t^{(k)}_1\<,\:t^{(k)}_2)$\,:
\vvn.2>
\begin{align}
\label{ind}
W_I\%{k+1}(\ttt\%k\<,\:t^{(k+1)}_1\<,\:t^{(k+1)}_2)\, &{}={}
\\[4pt]
{}=\,\Sym_{t^{(k)}_1\!,\>t^{(k)}_2}\Bigl(
W_{s(I)}\%k(\ttt\%{k-1}\<,\:t^{(k)}_1\<,\:t^{(k)}_2\!\< & \:\,{})\>
\bigl(\>\rho(t^{(k)}_2\!/\:t^{(k)}_1)\;
\zt(t_2^{(k)}\?,\:t_1^{(k)}\?,\:t_1^{(k+1)}\?,\:t_2^{(k+1)}\?,
\:\mu_b\:/\mu_a\:,\>\mu_{k+1}/\mu_b)\:-{}
\notag
\\[4pt]
& {}-\:\phi(t^{(k)}_1\!/\:t^{(k)}_2)\;
\zt(t_1^{(k)}\?,\:t_2^{(k)}\?,\:t_1^{(k+1)}\?,\:t_2^{(k+1)}\?,
\:\mu_b\:/\mu_a\:,\>\mu_{k+1}/\mu_b)\bigr)\?\Bigr)\,.
\notag
\\[-15pt]
\notag
\end{align}
The next step is to use the identity
\vvn.2>
\begin{align*}
& \rho(x_1/x_2)\,\zt(x_1,x_2,y_1,y_2,\al_1,\al_2)\:-\:
\phi(x_2/x_1)\,\zt(x_2,x_1,y_1,y_2,\al_1,\al_2)\,={}
\\[4pt]
&{}=\,\rho(y_1/y_2)\,\zt(x_2,x_1,y_2,y_1,1/\al_1,\al_1\:\al_2)\:-\:
\phi(y_1/y_2)\,\zt(x_2,x_1,y_1,y_2,1/\al_1,\al_1\:\al_2)\,,
\\[-14pt]
\end{align*}
equivalent to~\Ref{three}\:, to get
\vvn.5>
\begin{alignat*}2
&W_I\%{k+1}(\ttt\%k\<,\:t^{(k+1)}_1\<,\:t^{(k+1)}_2)\,={}
\\[4pt]
& {}=\,\Sym_{t^{(k)}_1\!,\>t^{(k)}_2}
W_{s(I)}\%k(\ttt\%{k-1}\<,\:t^{(k)}_1\<,\:t^{(k)}_2)\,
\rho(t^{(k+1)}_1\!/\:t^{(k+1)}_2)\;
\zt(t_1^{(k)}\?,\:t_2^{(k)}\?,\:t_2^{(k+1)}\?,\:t_1^{(k+1)}\?,
\:\mu_a\:/\mu_b\:,\>\mu_{k+1}/\mu_a)\:-{}
\\[4pt]
&{}\>-\,\Sym_{t^{(k)}_1\!,\>t^{(k)}_2}
W_{s(I)}\%k(\ttt\%{k-1}\<,\:t^{(k)}_1\<,\:t^{(k)}_2)\,
\phi(t^{(k+1)}_1\!/\:t^{(k+1)}_2)\;
\zt(t_1^{(k)}\?,\:t_2^{(k)}\?,\:t_1^{(k+1)}\?,\:t_2^{(k+1)}\?,
\:\mu_a\:/\mu_b\:,\>\mu_{k+1}/\mu_a)\bigr)\,.
\\[-12pt]
\end{alignat*}
Applying formulae~\Ref{Utsk} and~\Ref{WWU} in the right\:-hand side of the last
equality transforms it into formula~\Ref{RW} with \,$k$ \,replaced by
\,$k+1$\,, which completes the induction step.

\vsk.2>
Theorem~\ref{thm:recur} is proved.
\end{proof}

\subsection{Transformation properties}

\begin{lem}
\label{lem:tr}
For \,$I\in\Il$\,, denote
\vvn.3>
\beq
\label{Sl}
G_\bla(\ttt,\zz,h,\mub)\,=\,\prod_{k=1}^{N-1}\,\prod_{a=1}^{\la^{(k)}}\,
\prod_{c=1}^{\la^{(k+1)}}\thi(t^{(k+1)}_c\?/\:t^{(k)}_a)\,\prod_{k=1}^{N-1}\,
\frac{\thi(t^{(k)}_1\!\dots t^{(k)}_{\la^{(k)}}\>h^{-\la_k}\mu_k/\mu_{k+1})}
{\thi(t^{(k)}_1\!\dots t^{(k)}_{\la^{(k)}})\;\thi(h^{-\la_k}\mu_k/\mu_{k+1})}
\;,\kern-1em
\eeq
\vv.1>
\begin{align}
\label{Si}
G_I(\zz,h,\mub)\, &{}=\,\prod_{j=1}^{N-1}\,
\frac{\thi(h^{-\la_j\lsym-\:\la_{N-1}}\mu_j/\mu_N)\;
\thi\bigr(\prod_{a\in I_j\!} z_a\bigr)}
{\thi\bigl(h^{-\la_j\lsym-\la_{N-1}}(\mu_j/\mu_N)\prod_{a\in I_j\!} z_a\bigr)}
\\[3pt]
&{}\>\times\,\prod_{j=1}^{N-1}\,\prod_{a\in I_j\!}\;\frac{\thi(z_a)}
{\thi(z_a\:h^{\>a\:-1-\:p_{I\<\<,\:j}(a)-\la_1\lsym-\:\la_{j-1}})}\;,
\notag
\\[-14pt]
\notag
\end{align}
where $\;t^{(N)}_c\!=z_c$ \>and
$\;p_{I\<\<,\:j}(a)=|\,I_j\cap\{1\lc a-1\}\>|$\,, see~\Ref{pjk}. Set
\vvn.3>
\beq
\label{trpr}
G_{\bla,I}(\ttt,\zz,h,\mub)\,=\,
G_\bla(\ttt,\zz,h,\mub)\,G_I(\zz,h,\mub)\,.
\vv.3>
\eeq
Then the ratio
\;$W^{\>\on{ell}}_I(\ttt,\zz,h,\mub)/G_{\bla,I}(\ttt,\zz,h,\mub)$ \,does not
change if any of the variables \;$t^{(k)}_i\<,\:h,\:\mu_j$ \,is multiplied by
\;$q$\,.
\end{lem}
\begin{proof}
By formula \Ref{UI}, the ratio
\;$U^{\>\on{ell}}_I(\ttt,\zz,h,\mub)/G_{\bla,I}(\ttt,\zz,h,\mub)$ \,does not
change if any of the variables \;$t^{(k)}_i\<,\:h,\:\mu_j$ \,is multiplied
by \,$q$\,. Since \,$G_{\bla,I}(\ttt,\zz,h,\mub)$ \,is symmetric in
\,$t^{(i)}_1\!\lc t^{(i)}_{\la^{(i)}}$ \,for each \,$i=1\lc N-1$\,,
the same is true for
\;$W^{\>\on{ell}}_I(\ttt,\zz,h,\mub)/G_{\bla,I}(\ttt,\zz,h,\mub)$\,.
\end{proof}

Notice that
\begin{align*}
& a-1-p_{I\<\<,\:j(a)}-\la_1\lsym-\la_{j-1}\,={}
\\[4pt]
&{}\!\!=\,\bigl|\:\{1\lc a-1\}\cap\bigl(I_{j+1}\lsym\cup I_N\bigr)\:\bigr|
\>-\>\bigl|\:\{a+1\lc n\}\cap\bigl(I_1\lsym\cup I_{j-1}\bigr)\:\bigr|
\end{align*}

\vsk.3>\noindent
Let $\;\nu_k=h^{\>\la_1\lsym+\:\la_{k-1}}\mu_k$\,.
Then $\;h^{-\la_k}\mu_k/\mu_{k+1}=\:\nu_k/\nu_{k+1}$ \,and
$\;h^{-\la_j\lsym-\:\la_{N-1}}\mu_j/\mu_N=\:\nu_j/\nu_N$\,.

\vsk>
\let\exno1
\begin{exam}
$\,N=2\,,\;I_1=\{\:s\:\}\,,\;\la_1=1\,,\;\la_2=n-1$\,.
\vvn.1>
\be
W_I(t,\zzz)\,=\,\prod_{a=1}^{s-1}\,\thi(h\:z_a/t)\;\;
\frac{\thi\bigl(h^{\:2-s}\>(\mu_1/\mu_2)\>z_s/t\bigr)}
{\thi(h^{\:2-s}\:\mu_1/\mu_2)}\,\prod_{b=s+1}^n\!\thi(z_b/t)\;.
\vv.1>
\ee
The product~\Ref{trpr} equals
\be
\prod_{a=1}^n\,\thi(z_a/t)\;\;
\frac{\thi(h^{-1}t\:\mu_1/\mu_2)}{\thi(h^{-1}\mu_1/\mu_2)\,\thi(\:t\:)}\;
\frac{\thi(h^{-1}\mu_1/\mu_2)\,\thi(z_s)}{\thi(h^{-1}\:z_s\:\mu_1/\mu_2)}\;
\frac{\thi(z_s)}{\thi(z_sh^{\>s-1})}
\,\;\prod_{b=1}^{s-1}\,\frac{\thi(z_b)}{\thi(z_b\:h^{-1})}
\;\prod_{b=s+1}^n\:\frac{\thi(z_b)}{\thi(z_b)}\;,
\vv.2>
\ee
and $\;\nu_1=\:\mu_1\,,\;\nu_2=h\:\mu_2\,,\;h^{-1}\mu_1/\mu_2=\:\nu_1/\nu_2$\,.
\end{exam}

\vsk.3>
\let\exno2
\begin{exam}
$\,N=2\,,\;n=3\,,\;I_1=\{1,2\:\}\,,\;I_2=\{\:3\:\}\,,
\;\la_1=2\,,\;\la_2=0$\,.
\vvn.4>
\begin{align*}
W_I(t_1,t_2,z_1,z_2,z_3)\, &{}=\,
\thi(z_2/t_1)\;\thi(z_3/t_1)\;\thi(h\>z_1/t_2)\;\thi(z_3/t_2)\,\times{}
\\[6pt]
& {}\times\;\frac{\thi\bigl(h\>(\mu_2/\mu_1)\>z_1/t_1\bigr)\;
\thi\bigl(h^2(\mu_2/\mu_1)\>z_2/t_2\bigr)}
{\thi(h\:\mu_2/\mu_1)\;\thi(h^2\mu_2/\mu_1)}\;
\frac{\thi(h\>t_2/t_1)}{\thi(t_2/t_1)}\;+\;[\>t_1\leftarrow\!\!\to\:t_2\>]\;.
\end{align*}
The product~\Ref{trpr} equals
\vvn.4>
\be
\prod_{a=1}^2\,\prod_{c=1}^3\,\thi(z_c/t_a)\;\times\,
\frac{\thi(t_1\:t_2\>h^{-2}\mu_1/\mu_2)}
{\thi(t_1\:t_2)\;\thi(h^{-2}\mu_1/\mu_2)}\,\times\,
\frac{\thi(h^{-2}\mu_1/\mu_2)\;\thi(z_1z_2)}
{\thi(h^{-2}z_1z_2\>\mu_1/\mu_2)}\,\times\,\frac{\thi(z_1)}{\thi(z_1)}\;
\frac{\thi(z_2)}{\thi(z_2)}\;\frac{\thi(z_3)}{\thi(z_3)}\;,
\vv.4>
\ee
and $\;\nu_1=\:\mu_1\,,\;\nu_2=h^2\mu_2\,,\;h^{-2}\mu_1/\mu_2=\:\nu_1/\nu_2$\,.
\end{exam}
\vsk.4>
\noindent
Trivial fractions from~\Ref{trpr} are not removed in the Examples.

\subsection{Substitution properties of elliptic weight functions}
\label{sec:e-sub}

Recall
\vvn.1>
that for \,$I\?\in \Il$ \,we use the notation
\;$\bigcup_{a=1}^k I_a=\:\{\>i^{(k)}_1\!\lsym<i^{(k)}_{\la^{(k)}}\}$\,.
For a function \,$f(\ttt,\zz,h)$\,, we denote by \,$f(\zz_I,\zz,h)$
\,the result of the substitution \;$t^{(k)}_a\!=z_{i^{(k)}_a}$
\,for all \,$k=1\lc N$, \,$a=1\lc\la^{(k)}$.

\vsk.2>
For any \,$\si\in S_n$, \,we define the {\it combinatorial\/}
partial ordering \;$\leq_{\:\si}$ \,on \,$\Il$\,. For \,$I,J\in\Il$\,, \,let
\vvn.2>
\be
\si^{-1}\bigl(\>\tbigcup_{m=1}^k I_m\bigr)\,=\,
\{\>a_1^{(k)}\lsym<a_{\la^{(k)}}^{(k)}\}\,,\qquad
\si^{-1}\bigl(\>\tbigcup_{\ell=1}^k J_\ell\bigr)\,=\,
\{\>b_1^{(k)}\lsym<b_{\la^{(k)}}^{(k)}\}\,,
\vv.2>
\ee
for \,$k=1\lc N-1$\,. We say that \,$J\leq_{\:\si}\?I$ \,if
\vvn.1>
\;$b_i^{(k)}\leq a_i^{(k)}$ \,for all \,$k=1\lc N-1$, \;$i=1\lc\la^{(k)}$.
Notice that \,$J\leq_{\:\si}\?I$ \,if and only if
\,$\si^{-1}(J)\leq_{\>\id}\<\si^{-1}(I)$\,.

\vsk.2>
Set
\vvn-.6>
\beq
\label{Ela}
E^{\>\on{ell}}_\bla(\ttt,h)\,=\,\prod_{k=1}^{N-1}\,\prod_{a=1}^{\la^{(k)}}\,
\prod_{b=1}^{\la^{(k)}}\,\thi(h\:t^{(k)}_b\?/{t^{(k)}_a})
\vv-.4>
\eeq
and
\vvn.2>
\beq
\label{Wti}
\Wt^{\>\on{ell}}_{\si\<,\>I}(\ttt,\zz,h,\mub)\,=\,
\frac{W^{\>\on{ell}}_{\si\<,\>I}(\ttt,\zz,h,\mub)}
{E^{\>\on{ell}}_\bla(\ttt,h)}\;.
\vv.5>
\eeq

The proofs of the following three lemmas are straightforward modifications of
the proofs in \cite[Sect. 6.1]{RTV2}. We indicate main steps of the proof after
formulating the lemmas.

\begin{lem}
\label{lem:tria}
We have \;$\Wt^{\>\on{ell}}_{\si\<,\>I}(\zz_J,\zz,h,\mub)\:=\:0$
\,unless \;$J\leq_{\:\si}\?I$\,.
\end{lem}

Define
\beq
\label{PIE}
P^{\>\on{ell}}_{\si\<,\>I}(\zz,h)\,=
\,\prod_{k<l}\,\prod_{\si(a)\in I_k}\!
\biggl(\>\prod_{\satop{\si(b)\in I_l}{b<a}}\!\thi(hz_{\si(b)}/z_{\si(a)})\!
\prod_{\satop{\si(b)\in I_l}{b>a}}\!\thi(z_{\si(b)}/z_{\si(a)})\biggr).
\eeq

\begin{lem}
\label{lem:expec}
We have \;$\Wt^{\>\on{ell}}_{\si\<,\>I}(\zz_I,\zz,h,\mub)\:=\:
P^{\>\on{ell}}_{\si\<,\>I}(\zz,h)$\,.
\end{lem}

Recall the function \,$\psi_I(h,\mub)$\,, see \Ref{psI}.
For $J\in\Il$\,, denote
\vvn.3>
\beq
\label{evert}
e^{\on{ell},\vert}_{\si\<,\>J\<,-}(\zz,h)\,=\,
\prod_{k<l}\,\prod_{\si(a)\in J_k}\<\prod_{\satop{\si(b)\in J_l}
{b<a}}\!\thi(hz_{\si(b)}/z_{\si(a)})\,.\kern-.6em
\eeq

\begin{lem}
\label{lem:div_e}
For all \;$I,J\in\Il$\,, the function
$\psi_I(h,\mub)\>\Wt^{\>\on{ell}}_{\si\<,\>I}(\zz_J,\zz,h,\mub)\:
/e^{\on{ell},\vert}_{\si\<,\>J\<,-}(\zz,h)$
\vvn.1>
\,lifts to a regular function on the universal cover of
\;$(\C^\times)^{\:n+1+N}$.
\end{lem}

The universal cover in the last lemma is due to presence of the square root
\,$x^{1/2}$ in formula \Ref{sqr} for \,$\thi(x)$\,.

\begin{proof}[Proof of Lemmas \ref{lem:tria}--\ref{lem:div_e}]
By formula \Ref{Wsi}\:, it suffices to verify the claims for \,$\si$
\vvn.1>
\,being the identity permutation. The function
\;$W^{\>\on{ell}}_{\id\<,\>I}(\ttt,\zz,h,\mub)$ is the symmetrization
of product \Ref{UI}. By inspection of formula \Ref{UI}, each term in
the symmetrization sum acquire a zero factor at \,$\TT=\zz_J$ \,unless
\,$J\le_{\>\id} I$\,, and \,$E^{\>\on{ell}}_\bla(\zz_J,h)\ne0$\,.
For \,$J=I$, the only term of the sum for
\;$W^{\>\on{ell}}_{\id\<,\>I}(\ttt,\zz,h,\mub)$ that does not vanish
at \,$\TT=\zz_J$ \,corresponds to the identity element, and equals
$U^{\>\on{ell}}_I(\ttt,\zz,h,\mub)$\,. Then
\vvn.1>
\;$\Wt^{\>\on{ell}}_{\id\<\>,\>I}(\zz_I,\zz,h,\mub)\:=\:
U^{\>\on{ell}}_{\id\<\>,\>I}(\zz_I,\zz,h,\mub)/E^{\>\on{ell}}_\bla(\zz_I,h)
\:=\:P^{\>\on{ell}}_{\id\<\>,\>I}(\zz,h)$\,. Finally for \,$J<_{\>\id} I$\,,
inspection of \Ref{UI} shows that each potentially nonzero term of the sum
for \;$W^{\>\on{ell}}_{\id\<,\>I}(\zz_J,\zz,h,\mub)$ is divisible by
\,$E^{\>\on{ell}}_\bla(\zz_I,h)\>e^{\on{ell},\vert}_{\si\<,\>J\<,-}(\zz,h)$\,.
\end{proof}

\subsection{Orthogonality}
Let \,$\si_0\in$ be the longest permutation in \,$S_n$\>. \>Denote
\;$h^\bla\mub^{-1}\<=(h^{\la_1}\<\mu_1^{-1}\!\lc h^{\la_N}\<\<\mu_N)^{-1}$\>.
Let \,$\tti$ \,be an additional set of variables similar to \,$\TT$\,.
Define a function \,$\Xi(\TT,\tti,\zz,h,\mub)$\,,
\vvn.3>
\beq
\label{Xi}
\Xi(\TT,\tti,\zz,h,\mub)\,=\>
\sum_{I\in\Il}\,\Wt^{\>\on{ell}}_{\id\:,\:I}(\TT,\zz,h,\mub)\,
\Wt^{\>\on{ell}}_{\si_0,\:I}(\tti,\zz,h,h^\bla\mub^{-1})\,.
\eeq

\begin{prop}
\label{Xiz}
The function \;$\Xi(\TT,\tti,\zz,h,\mub)$ is symmetric in \,$\zz$\,.
\end{prop}
\begin{proof}
The statement follows from Theorem \ref{thm:recur}.
\end{proof}

\begin{prop}
\label{XiIJp}
For \,$I\:,J\in\Il$, we have
\vvn.2>
\beq
\label{XiIJ}
\Xi(\zz_I,\zz_J,\zz,h,\mub)\,=\,
R^{\>\on{ell}}(\zz_I)\,Q^{\>\on{ell}}(\zz_I,h)\,\dl_{I\<,\:J}\,,
\eeq
where
\beq
\label{RQ}
R^{\>\on{ell}}(\zz_I)\,=\,
\prod_{k<l}\,\prod_{a\in I_k}\,\prod_{b\in I_l}\,\thi(z_b/z_a)\,,
\qquad
Q^{\>\on{ell}}(\zz_I,h)\,=\,
\prod_{k<l}\,\prod_{a\in I_k}\,\prod_{b\in I_l}\,\thi(hz_b/z_a)\,.\kern-2em
\vv.3>
\eeq
\end{prop}
\begin{proof}
Denote by \,$f_{I\<,\<\>J}(\zz)$ one of the sides of formula \Ref{XiIJ}.
Then either by Proposition \ref{Xiz} or by formulae \Ref{RQ}, we have
\,$f_{\si(I),\>\si(J)}(\zz)=f_{I\<,\<\>J}(\zzzsi)$ \,for any \,$\si\in S_n$\,.
\vvn.1>
Thus it is enough to verify formula \Ref{XiIJ} for
\vvn.1>
\,$J=(\{1\lc\la_1\}\>,\,\ldots\,,\{n-\la_n+1\lc n\})$\,.
In this case, \,$I\<\le_{\si_0}\!J$ \,for any \,$I\in\Il$\,, and
formula \Ref{XiIJ} follows from Lemmas \ref{lem:tria}, \ref{lem:expec}.
\end{proof}

\begin{thm}
\label{thm orth}
For \,$J\:,K\!\in\Il$, we have
\vvn.4>
\beq
\label{ORT}
\sum_{I\in\Il}\,\frac{\Wt^{\>\on{ell}}_{\id\:,\:J}(\zz_I,\zz,h,\mub)\,
\Wt^{\>\on{ell}}_{\si_0,\:K}(\zz_I,\zz,h,h^\bla\mub^{-1})}
{R^{\>\on{ell}}(\zz_I)\,Q^{\>\on{ell}}(\zz_I,h)}\,=\,\dl_{J,K}\,.
\eeq
\end{thm}
\begin{proof}
Consider the matrices $\;\Wh,\,\Wh'\<\<,\>\Qh$ \,with entries
\vvn.4>
\be
\Wh_{I\<,\:J}=\Wt^{\>\on{ell}}_{\id\:,\:J}(\zz_I,\zz,h,\mub)\,,\quad
\Wh'_{K\<,\:I}=\Wt^{\>\on{ell}}_{\si_0,\:K}(\zz_I,\zz,h,h^\bla\mub^{-1})\,,
\quad \Qh_{I\<,\:J}=R^{\>\on{ell}}(\zz_I)\,Q^{\>\on{ell}}(\zz_I,h)\,
\dl_{I\<,\:J}\,.\kern-.4em
\vv.2>
\ee
Formula \Ref{XiIJ} yields a matrix equality $\;\Wh\>\Wh'\<=\:\Qh$\,.
Therefore, $\;\Wh'\>\Qh^{-1}\>\Wh=\:1$\,, which is exactly formula \Ref{ORT}.
\end{proof}

\begin{rem}
Denote \;$\ttt^{-1}\<=\bigl((t^{(1)}_1)^{-1}\!\lc
(t^{(N-1)}_{\la^{(N-1)}})^{-1}\bigr)\,,\;\;
\zz^{-1}\<=(z_1^{-1}\<\lc z_n^{-1})$\,. Notice that all the consideration
is essentially invariant under the simultaneous replacement
\be
h\to h^{-1}\<,\quad \zz\to\zz^{-1}\<,\quad \ttt\to\ttt^{-1}\<,\quad
\mub\to\mub^{-1}\<,
\ee
for instance, \,$\Rc(x,h,\mub)\,=\,\Rc(x^{-1}\<,h^{-1}\<,\mub^{-1})$\,,
\be
W^{\>\on{ell}}_{\si\<,\>I}(\ttt^{-1}\<,\zz^{-1}\<,h^{-1}\<,\mub^{-1})\,=\,
(-1)^{\sum_{k=1}^{N-1} \la^{(k)}\la^{(k+1)}}\,W^{\>\on{ell}}_{\si\<,\>I}(\ttt,\zz,h,\mub)\,.
\ee
\end{rem}

\section{Trigonometric weight functions}
\label{sec:trigweight}
\subsection{Alcoves}
Consider \,$\R^N$ with coordinates $\nu=(\nu_1,\dots,\nu_N)$.
A {\it wall\/} in \,$\R^N$ is a hyperplane defined by an equation of the form
$\nu_j\<-\nu_i=\:m$, \,$1\leq i<j\leq N$, \,$m\in\Z$. Connected components
of the complement in \,$\R^N$ to the union of all walls are called
{\it alcoves}. Let \,$\Dl$ \,be an alcove and \,$\nu\in\Dl$\,.
For $1\leq i<j\leq N$, \,denote $m_{i,j}\<=\>\lfloor\nu_j\<-\nu_i\rfloor\in\Z$,
that is, $0<\nu_j\<-\nu_i\<-m_{i,j}\<<1$. The numbers
\,$(m_{i,j})_{1\leq i<j\leq N}$ \>depend on the alcove \,$\Dl$\,,
but not on \,$\nu\in\Dl$\,; they are called the {\it integer characteristics\/}
of the alcove $\Dl$\,. The {\it anti-dominant\/} alcove is the alcove with zero
integer characteristics.

\vsk.2>
Let $\Dl$ be an alcove with integer characteristics $(m_{i,j})$.
Recall the $\ttt,\zz, h$ variables from Section \ref{sec:ellwe}.
For $I\in\Il$, define the {\it trigonometric weight function\/}
\vvn.2>
\beq
\label{WItr}
W_I^\Dl (\ttt,\zz,h)\,=\,\bigl(1-h\bigr)^{\la^{\{1\}}}\<
\Sym_{\>t^{(1)}} \ldots \Sym_{\>t^{(N-1)}} U^\Dl_I\,,
\eeq
where
\vvn.2>
\be
\Sym_{\>t^{(k)}} f\bigl(t^{(k)}_1\lc t^{(k)}_{\la^{(k)}}\bigr)\,=\!
\sum_{\si\in S_{\la^{(k)}}\!\!}\<
f\bigl(t^{(k)}_{\si(1)}\lc t^{(k)}_{\si(\la^{(k)})}\bigr)\,,
\vv-.1>
\ee
\beq
\label{UItr}
U^\Dl_I\,=\,\prod_{k=1}^{N-1}\,\prod_{a=1}^{\la^{(k)}}\,\biggl(\,
\prod_{c=1}^{\la^{(k+1)}}\,\psi^\Dl_{I,k,a,c}(t^{(k+1)}_c/t^{(k)}_a)\,
\prod_{b=a+1}^{\la^{(k)}}\?\frac{1-ht^{(k)}_b/t^{(k)}_a}
{1-t^{(k)}_b/t^{(k)}_a}\,\biggr)\>,
\vv.5>
\eeq
\beq
\label{psitr}
\psi^\Dl_{I,k,a,c}(x)\,=\left\{\begin{alignedat}3
&\kern1.25em 1-hx\,,&&\text{if} && i^{(k+1)}_c\!<i^{(k)}_a\,,\\[2pt]
&\,x^{-m_{{j(I\<,\>k,\:a)},k+1}}\>,\;\quad &&
\text{if}\;\quad&& i^{(k+1)}_c\!=i^{(k)}_a\,,
\\[2pt]
&\kern1.5em\,1-x\,,&&\text{if} && i^{(k+1)}_c\!>i^{(k)}_a\,.
\end{alignedat}\right.\kern-1.6em
\eeq
where \,${j(I\<,\>k,\:a)}\in\{1,\dots,N\}$ \,is such that
\,$i^{(k)}_a\<\in\<I_{j(I\<,\>k,\:a)}$\,, \,cf.~\Ref{psi}\:.

\vsk.3>
For \,$\si\in S_n$ \>and \,$I\?\in\Il$\>, define the trigonometric weight
function
\vvn.4>
\beq
\label{Wsitr}
W^\Dl_{\si\<,\>I}(\ttt,\zz, h)\,=\,
W^\Dl_{\si^{-1}(I)}(\ttt,z_{\si(1)}\lc z_{\si(n)},h)\,,
\vv.2>
\eeq
where \;$\si^{-1}(I)=\bigl(\si^{-1}(I_1)\lc\si^{-1}(I_N)\bigr)$.
Hence, \,$W^\Dl_I=W^\Dl_{\id,I}$.

\vsk.2>
Define
\vvn-.6>
\beq
\label{trigE}
E(\ttt,h)=
\prod_{k=1}^{N-1}\prod_{a=1}^{\la^{(k)}} \prod_{b=1}^{\la^{(k)}}
\left(1-\frac{ht^{(k)}_b}{t^{(k)}_a}\right)
\vv-.5>
\eeq
and
\vvn-.2>
\beq
\label{Wtilde}
\Wt^\Dl_{\si\<,\>I}(\ttt,\zz,h)=
\frac{W^\Dl_{\si\<,\>I}(\ttt,\zz,h)}{E(\ttt,h)}.
\vv.1>
\eeq

\begin{rem}
The trigonometric weight functions of \cite{RTV2} correspond to
\,$W^\Dl_{\si\<,\>I}(\ttt,\zz,h)$ \,with the anti-dominant alcove \,$\Dl$\,.
\end{rem}

\subsection{Elliptic versus trigonometric weight functions: the $q\to 0$ limit}
\label{qto0}
Let \,$\Dl$ \,be an alcove with integer characteristics \,$(m_{j\<,\:k})$\,,
\,and let \,$\nu\in\Dl$\,. Define
\vvn.2>
\beq
\label{lim nu}
\mu_j\:=\,e^{2\pi i \tau \nu_j}\>, \qquad j=1\lc N\:.
\vv.1>
\eeq
Then for \,$1\leq j<k\leq N$, \,we have
\,$\mu_k/\mu_j=\:q^{\>m_{j,k}+\eps_{j,k}}$,
where \;$\eps_{j\<,\:k}=\nu_k\<-\nu_j\<-m_{j\<,\:k}\in(0,1)$\,.

\vsk.2>
For \,$I\<\in\Il$\,, define
\beq
\label{ellI}
\ell_I\:=\,\#\{\:(a,b)\in\{1\lc n\}^2\ \,|\ \;
a\in I_j\>,\,\,b\in I_k\>,\,\,a>b\,,\,\,j<k\>\}\,.
\vv.2>
\eeq
Equivalently, \;$\ell_I$ \,is the length of the permutation \;$\si_I\<\in S_n$
\>such that
\vvn.1>
\beq
\label{siI}
I_j\>=\,\{\>\si_I(\la^{(j-1)}\?+1)\lsym<\si_I(\la^{(j)})\:\}\,,\qquad j=1\lc N.
\eeq
\begin{thm}
\label{pr4.1}
Let \;$h,\zz,\ttt$ \,be fixed. As \;$q\to 0$\,, \,we have
\vvn.1>
\beq
\label{q0}
W_I^{\>\on{ell}}(\ttt,\zz,h,\mub)\,\to\,c(h)\,h^{-\ell_I\</2}\>
\bigl(X(\ttt)\bigr)^{-1/2}\,W^\Dl_I(\ttt,\zz,h)\,,
\eeq
where \;$c(h)\,=\,(-1)^{\sum_{k=1}^{N-1} \la^{(k)}\la^{(k+1)}}\>
h^{-\?\sum_{k=1}^{N-1}\>(\la^{(k)})^2\!/2}$ \>and
\be
X(\ttt)\,=\,\prod_{k=1}^{N-1}\,\prod_{a=1}^{\la^{(k)}}
\prod_{b=1}^{\la^{(k+1)}}(t^{(k+1)}_b\!/t^{(k)}_a)\,=\,
\prod_{k=1}^{N-1}\,\prod_{a=1}^{\la^{(k)}}\,
(t^{(k)}_a)^{-\:\la_k\<-\la_{k+1}}\,
\prod_{a=1}^n z_a^{n\:-\la_N}
\ee
\end{thm}
\begin{proof}
The statement follows from formulae \Ref{WI}\,--\,\Ref{pjk} and
\Ref{WItr}\,--\,\Ref{psitr} for the weight functions, and the limit
properties~\Ref{eqn:thetl}\:. The next formulae simplify the expression
for \,$c(h)$\,:
\vvn-.4>
\be
\sum_{a=1}^{\la^{(k)}}\,\sum_{c=1}^k\,p_{I\<\<,\:c}(i^{(k)}_a)\>=\,
\frac{\la^{(k)}\>(\la^{(k)}\?-1)}2\;,\qquad
\sum_{k=1}^{N-1}\,\sum_{a=1}^{\la^{(k)}}\,p_{I\<\<,\:k+1}(i^{(k)}_a)
\>=\,\ell_I\,.\kern-1em
\vv->
\ee
\end{proof}

\begin{rem}
In the geometric language of Section~\ref{Preliminaries from Geometry},
\,$\ell_I$ \>equals the dimension of the Schubert cell \;$\Om_{\:\id\:,\:I}$\>,
\vvn.1>
see~\Ref{OmsI}\:, \Ref{ellsI}\:. Thus the ratio of the factor
\;$h^{-\ell_I\</2}$ \>in \Ref{q0} and the factor
\;$h^{\codim(\Om_{\id,I}\subset \Fla)/2}$ \>in the normalization condition
in Section~\ref{sec:5} equals \;$h^{-(\dim\Fla)/2}$ and does not depend on
\,$I\<\in\Il$\,.
\end{rem}

\begin{rem}
Notice the substitution formula, cf.~\Ref{Pe}\::
\be
X(\zz_I)\,=\,\prod_{1\le j<k\le N}\,\prod_{a\in I_j\!}\,\prod_{b\in I_k\!}
\;(z_b/z_a)\,.
\ee
\end{rem}

\begin{rem}
Theorem \ref{pr4.1} is our attempt to interpret \cite[Proposition 3.5]{AO}
as a property of the elliptic and trigonometric weight functions.
In \cite[Proposition 3.5]{AO} the authors claim for a Nakajima variety that
an elliptic stable envelope converges to a certain $K\<$-theoretic stable envelope
under a certain limit. The $K\<$-theoretic stable envelope are not precisely
defined in \cite{AO} and the proof of Proposition 3.5 in \cite{AO} is only
sketched.
\end{rem}

\subsection{Recursion properties}
Let \;$s_{a,a+1}(\zz)=(z_1\lc z_{a-1},z_{a+1},z_a,z_{a+2}\lc z_n)$\,.

\begin{thm}
\label{thm:r_tri}
Consider an alcove \,$\Dl$ \>with integer characteristics \,$(m_{i,j})$\,,
and \,$I\in\Il$ \,such that \;$a\in I_k\>,\,a+1\in I_l$\,. If \;$k=l$\,, then
\vvn.2>
\beq
\label{e1}
W^\Dl_I(\ttt,s_{a,a+1}(\zz),h)=W^\Dl_I(\ttt,\zz,h).
\vv.1>
\eeq
If \;$k<l$\,, then
\vvn.1>
\begin{align}
\label{e2}
& W^\Dl_{s_{a,a+1}(I)}(\ttt,\zz,h)\,={}
\\[4pt]
& \!{}=\;\frac{1-hz_{a}/z_{a+1}}{1-z_a/z_{a+1}}\;W^\Dl_I(\ttt,s_{a,a+1}(\zz),h)
\:+\:(h-1)\,\frac{(z_a/z_{a+1})^{m_{k,l}+1}}{1-z_a/z_{a+1}}\;
W^\Dl_I(\ttt,\zz,h)\,.\kern-1em
\notag
\\[-17pt]
\notag
\end{align}
If \;$k>l$\,, then
\vvn.1>
\begin{align}
\label{e3}
& W^\Dl_{s_{a,a+1}(I)}(\ttt,\zz,h)\,={}
\\[4pt]
& {}\!=\;\frac{1-h^{-1}z_{a+1}/z_{a}}{1-z_{a+1}/z_{a}}\;
W^\Dl_I(\ttt,s_{a,a+1}(\zz),h)\:+\:
(h^{-1}-1)\,\frac{(z_{a+1}/z_a)^{m_{l,k}+1}}{1-z_{a+1}/z_a}\;
W^\Dl_I(\ttt,\zz,h)\,.\kern-1em
\notag
\end{align}
\end{thm}

\begin{proof}
The statement follows from Theorems~\ref{thm:recur} and~\ref{pr4.1}.
\end{proof}

\subsection{Substitution properties of trigonometric weight functions}
\label{sec:substitutions}

Recall
\vvn.1>
that for \,$I\?\in \Il$ \,we use the notation
\;$\bigcup_{a=1}^k I_a=\:\{\>i^{(k)}_1\!\lsym<i^{(k)}_{\la^{(k)}}\}$\,.
For a function \,$f(\ttt,\zz,h)$\,, we denote by \,$f(\zz_I,\zz,h)$
\,the result of the substitution \;$t^{(k)}_a\!=z_{i^{(k)}_a}$
\vvn.2>
\,for all \,$k=1\lc N$, \,$a=1\lc\la^{(k)}$.
Recall the partial ordering \;$\leq_{\:\si}$ \,on \,$\Il$ \,defined
in Section~\ref{sec:e-sub}.

\vsk.2>
The proofs of the following four lemmas are straightforward modifications
of those in \cite[Section~6.1]{RTV2}\:, cf.~the proof of
Lemmas~\ref{lem:tria}\,--\,\ref{lem:div_e} in Section~\ref{sec:e-sub}.

\begin{lem}
\label{lem:WK divisible by E}
For any \;$\si\in S_n$ and \,$I,J\in\Il$\,, the function
\;$\Wt^\Dl_{\si\<,\>I}(\zz_J,\zz,h)$ is a Laurent polynomial.
\end{lem}

\begin{lem}
\label{lem:triang}
We have \;$\Wt^\Dl_{\si\<,\>I}(\zz_J,\zz,h)=0$ \,unless \,$J\leq_{\:\si}\?I$\,.
\end{lem}

Define
\beq
\label{PsiI}
P_{\si\<,\>I}\,=\prod_{k<l}\prod_{\si(a)\in I_k}
\prod_{\satop{\si(b)\in I_l}{b>a}}\!\bigl(-\:z_{\si(b)}/z_{\si(a)}\bigr)\,.
\vv.3>
\eeq
Using definitions~\Ref{ehor}\:, \Ref{evertr} \,of
\,$e_{\si\<,\>I}=\:e_{\si\<,\>I\<,-}^{\hor}\>e_{\si\<,\>I\<,-}^{\vert}$\,,
observe that
\vvn.2>
\beq
\label{Pe}
P_{\si\<,\>I} \cdot e_{\si\<,\>I}=\,\prod_{k<l}\,\prod_{\si(a)\in I_k}\!
\biggl(\>\prod_{\satop{\si(b)\in I_l}{b<a}}\!(1-hz_{\si(b)}/z_{\si(a)})\!
\prod_{\satop{\si(b)\in I_l}{b>a}}\!(1-z_{\si(b)}/z_{\si(a)})\biggr).
\eeq

\begin{lem}
\label{lem:expected}
We have
\;$\Wt^\Dl_{\si\<,\>I}(\zz_I,\zz,h)=P_{\si\<,\>I}\cdot e_{\si\<,\>I}$\,.
\end{lem}

\begin{lem}
\label{lem:div_e_ver}
For any \,$I,J\in\Il$\,, the ratio
\;$\Wt^\Dl_{\si\<,\>I}(\zz_J,\zz,h)/e^{\vert}_{\si\<,\>J,-}$
is a Laurent polynomial.
\end{lem}

\subsection{Newton polytope properties} \label{sec:Newton}
For $a=1,\ldots,n-1$, denote by \,$K_a$ the automorphism of
\,$\C[\zz^{\pm1},h^{\pm 1}]$ \,switching \,$z_a$ and \,$z_{a+1}$\,.
That is \,$K_a(f(\zz))=f(s_{a,\:a+1}(\zz))$\,.

\vsk.2>
Let \;$\nu\in\R^N$ belong to the alcove \,$\Dl$ \,with integer characteristics
$(m_{i,\:j})$. We will use the notation
$\eps_{i,\:j}=\nu_j-\nu_i-m_{i,\:j}$ \,for the fractional part of
\;$\nu_j-\nu_i$ \>for \,$1\leq i<j\leq N$\:;
\vvn.1>
we have \,$0<\eps_{i,\:j}<1$\,. Define
\be
S_I=\>S_I(\zz)\,=\,\prod_{k=1}^N\,\prod_{a\in I_k\!}\,z_a^{-\nu_k}
\qquad\text{and}\qquad
\bW^\Dl_{I\<,\:J}\>=\>\bW^\Dl_{I\<,\:J}(\zz,h)\>=\>
W^\Dl_I(\zz_J,\zz,h)\cdot S_I(\zz)\,.
\kern-1em
\ee
Although the notation does not record it, both \,$S_I$ and
\,$\bW^\Dl_{I\<,\:J}$ \:depend on \;$\nu$ \,itself,
not only on its alcove \,$\Dl$\,.
% these two quantities

\begin{lem}
Let \,$I\in\Il$ be \,such that \;$a\in I_k\>,\,a+1\in I_l$\,.
If \;$k=l$\,, then
\vvn.3>
\beq\label{ee1}
\bW^\Dl_{I,s_{a,\:a+1}(J)}\,=\,K_a(\bW^\Dl_{I\<,\:J})\,.
\eeq
If \;$k<l$\,, then
\vvn-.1>
\beq
\label{ee2}
\bW^\Dl_{I,s_{a\<,a+1}(J)}\,=\,K_a\biggl(\:
\frac{1-z_{a}/z_{a+1}}{1-hz_a/z_{a+1}};\bW^\Dl_{s_{a\<,a+1}(I),J}\:+\:
(1-h)\,\frac{(z_a/z_{a+1})^{1-\eps_{k\<,\:l}}}{1-hz_a/z_{a+1}}\;
\bW^\Dl_{I\<,\:J}\:\biggr).\kern-1em
\eeq
If \;$k>l$\,, then
\beq
\label{ee3}
\bW^\Dl_{I,s_{a\<,a+1}(J)}\,=\,K_a\biggl(\:
\frac{1-z_{a+1}/z_{a}}{1-h^{-1}z_{a+1}/z_{a}}\;\bW^\Dl_{s_{a\<,a+1}(I),J}\:+\:
(1-h^{-1})\,\frac{(z_{a+1}/z_{a})^{1-\eps_{l,k}}}{1-h^{-1}z_{a+1}/z_{a}}\;
\bW^\Dl_{I\<,\:J}\biggr).\kern-1em
\eeq
\end{lem}
\vsk.2>
\begin{proof}
Let \,$k=l$\,. Then taking into account \,\Ref{e1}\:,
\vvn.3>
\begin{align*}
K_a(\bW^\Dl_{I\<,\:J})\,=\,K_a\bigl(W^\Dl_I(\zz_J,\zz,h)\bigr)\,K_a(S_I)\,&
{}=\,W^\Dl_I\bigl(\zz_{s_{a\<,a+1}(J)},s_{a,\:a+1}(\zz),h\bigr)\,S_I\,={}
\\[4pt]
&{}=\,W^\Dl_I(\zz_{s_{a\<,a+1}(J)},\zz,h)\,S_I\,=\,
\bW^\Dl_{I,s_{a\<,a+1}(J)}\,.
\end{align*}
Let \,$k<l$\,. Formula~\Ref{e2}\:, rearranged, gives
\vvn.3>
\be
W^\Dl_I\bigl(\ttt,s_{a,\:a+1}(\zz),h\bigr)\,=\,
\frac{1-z_a/z_{a+1}}{1-hz_a/z_{a+1}}\;W^\Dl_{s_{a\<,a+1}(I)}(\ttt,\zz,h)\:+\:
(1-h)\,\frac{(z_a/z_{a+1})^{m_{k,l}+1}}{1-hz_a/z_{a+1}}\;W^\Dl_I(\ttt,\zz,h)\,.
\kern-.5em
\vv.3>
\ee
Substitute now \,$\ttt=\zz_J$\,. Then,
\vvn.3>
\begin{align*}
K_a &{}\bigl(W^\Dl_I(\zz_{s_{a\<,a+1}(J)},\zz,h)\bigr)\,=\,
W^\Dl_I\bigl(\zz_J,s_{a,\:a+1}(\zz),h\bigr)\,={}
\\[4pt]
&{}\!=\;\frac{1-z_a/z_{a+1}}{1-hz_a/z_{a+1}}\;
W^\Dl_{s_{a\<,a+1}(I)}(\zz_{J},\zz,h)\:+\:(1-h)\,
\frac{(z_a/z_{a+1})^{m_{k,l}+1}}{1-hz_a/z_{a+1}}\;W^\Dl_I(\zz_{J},\zz,h)\,,
\\[-22pt]
\end{align*}
that is,
\vvn.2>
\be
\frac{K_a(\bW^\Dl_{I,s_{a\<,a+1}(J)})}{K_a(S_I)}\,=\,
\frac{1-z_a/z_{a+1}}{1-hz_a/z_{a+1}}\cdot
\frac{\bW^\Dl_{s_{a\<,a+1}(I),J}}{S_{s_{a\<,a+1}(I)}}\>+\>
(1-h)\,\frac{(z_a/z_{a+1})^{m_{k,l}+1}}{1-hz_a/z_{a+1}}\cdot
\frac{\bW^\Dl_{I,J}}{{S_I}}\;.
\vvgood
\vv.3>
\ee
Since \;$K_a(S_I)\:=\:S_{s_{a\<,a+1}(I)}\:=\:(z_{a+1}/z_a)^{\nu_l-\nu_k}\>S_I
\:=\:(z_{a+1}/z_a)^{m_{k\<,\:l}+\eps_{k\<,\:l}}\>S_I$\,, \,we obtain
\vvn.1>
\beq
\label{ee4}
K_a\bigl(\bW^\Dl_{I,s_{a\<,a+1}(J)}\bigr)\,=\,
\frac{1-z_a/z_{a+1}}{1-hz_a/z_{a+1}}\;\bW^\Dl_{s_{a\<,a+1}(I),J}\:+\:
(1-h)\,\frac{(z_a/z_{a+1})^{1-\eps_{k,l}}}{1-hz_a/z_{a+1}}\;\bW^\Dl_{I,J}\,,
\vv.3>
\eeq
which is equivalent to formula~\Ref{ee2}\:.
The case \,$k>l$ \,is proved similarly.
\end{proof}

For a Laurent polynomial \,$f\in\C[z_1^{\pm 1}\lc z_n^{\pm 1},\:h^{\pm 1}]$\,,
let \;$\Ne\:(f)$ \>be the convex hull of the finite set
\vvn-.2>
\be
\{\:(k_1\lc k_n)\in\Z^n\ |\ \,\text{the coefficient of }
\,z_1^{k_1}\<\dots z_n^{k_n}\,\text{ in }\,f\,\text{ is not }\,0\>\}\,.
\ee
That is, in Newton polytope considerations we treat \,$h$ \,as a number, not
as a variable. Let \,$O_{I\<,\:J}=\:\Ne\:\bigl(\bW^\Dl_{I\<,\:J}(\zz,h)\bigr)$\,.
The operators \,$K_a$ act naturally on Newton polytopes, satisfying
\,$K_a\bigl(\Ne\:(f(\zz))\bigr)=\:\Ne\:\bigl(f(s_{a,\:a+1}(\zz))\bigr)$.

\begin{lem}
\label{lem:extralem}
We have \;$O_{s_{a\<,a+1}(J),\>s_{a\<,a+1}(J)}\:=\:K_a(O_{J,\:J})$\,.
\end{lem}
\begin{proof}
Let \,$a\in J_k$, $a+1\in J_l$. If \,$k=l$\,, then the statement follows from
formula \Ref{ee1}\,.

\vsk.2>
Let \,$k\ne l$\,. Replacing \,$J$ \,by \;$s_{a\<,a+1}(J)$ \,if necessary,
we may assume that \,$k>l$\,. Denote \,$I=s_{a\<,a+1}(J)$\,.
We have $J\not\leq_{\>\id}\<I$, hence \,$\bW^\Dl_{I\<,\:J}=0$\,, \,and
formula~\Ref{ee3} yields
\vvn.3>
\be
(1-h^{-1}z_a/z_{a+1})\,\bW^\Dl_{s_{a\<,a+1}(J),s_{a\<,a+1}(J)}\,=\,
(1-z_a/z_{a+1})\,K_a(\bW^\Dl_{J,\:J})\,.
\vv.3>
\ee
Taking the Newton polytope of both sides, we obtain
\vvn.2>
\be
T+O_{s_{a\<,a+1}(J),s_{a\<,a+1}(J)}\,=\,T+K_a(O_{J,\:J})\,,
\vv.2>
\ee
where \;$+$ \;is the Minkowski sum and
\;$T=\:\Ne\:(1-h^{-1}z_a/z_{a+1})=\:\Ne\:(1-z_a/z_{a+1})$\,.
Using the cancellation law of Minkowski sum, we obtain the statement
of the lemma.
\end{proof}

The following theorem is one of the main results of this paper.

\begin{thm}
\label{thm:Newton}
We have \;$O_{I\<,\:J}\subset O_{J,\:J}$ \>for any \,$I,J\in\Il$\,.
\end{thm}

\begin{proof}
We will prove by induction on $J$ that \,$O_{I\<,\:J}\subset O_{J,\:J}$
\,for all \,$I\<\in\Il$\,. The induction employs the partial ordering
\,$\leq_{\>\id}$\,. The base of induction is
\be
J\>=\,\Imx=\,\bigl(\:\{\:n-\la_1\<+1\lc n\}\,,\>
\{\:n-\la_1\<-\la_2\<+1\lc n-\la_1\:\}\,,\,\ldots\,,\:\{\:1\lc\la_N\:\}
\:\bigr)\,.
\ee
Then \,$I\leq_{\>\id}\<J$ \,for all \,$I\<\in\Il$\,. Hence by
Lemma~\ref{lem:triang}, \,$\bW^\Dl_{I\<,\:J}=0$ \,for \,$I\ne J$\,,
and the statement holds.

\vsk.2>
If \,$J\ne\Imx$, let \,$a$ \,be such that $J<_{\>\id}\<s_{a\<,a+1}(J)=\Jt$\,.
Let \,$I\<\in\Il$ \,and define \,$k,l$ \,by the rule: \,$a\in I_k$\,,
$a+1\in I_l$\,. If \,$k=l$\,, then
\vvn.3>
\be
O_{I\<,\:J}\,=\,K_a(O_{I\<,\:\Jt})\subset K_a(O_{\Jt\<,\:\Jt})\,=\,O_{J,\:J}\,,
\vv.3>
\ee
where the first equality follows from formula \Ref{ee3}\:, the containment
holds by the induction assumption, and the last equality is due to
Lemma~\ref{lem:extralem}.

\vsk.2>
If \,$k<l$\,, let \,$O'$ be the Newton polytope of
\vvn.2>
\be
K_a\bigl(\bW^\Dl_{I,s_{a\<,a+1}(J)}\bigr)\,=\,
\frac{1-z_a/z_{a+1}}{1-hz_a/z_{a+1}}\;\bW^\Dl_{s_{a\<,a+1}(I),J}\:+\:
(1-h)\,\frac{(z_a/z_{a+1})^{1-\eps_{k,l}}}{1-hz_a/z_{a+1}}\;\bW^\Dl_{I,J}\,,
\vv.3>
\ee
see~(\ref{ee4})\:. Let \;$T=\:\Ne\:(1-z_a/z_{a+1})=\:\Ne\:(1-hz_a/z_{a+1})$
\,(an interval), and \,$p=\Ne\:\bigl((z_a/z_{a+1})^{1-\eps_{k,l}}\bigr)$
\,(one point)\:. From the definition of \,$O'$ we obtain
\vvn.3>
\be
T+O'=\,\chl\>\bigl(T+\Ne\:(\bW^\Dl_{s_{a\<,a+1}(I),\:\Jt})\:,
\,p\:+\Ne\:(\bW^\Dl_{I\<,\:\Jt})\bigr)\,,
\vv.3>
\ee
where \,$\chl$ \,stands for convex hull, and \;$+$ \;for the Minkowski sum.
Since \,$1-\eps_{k,l}\in(0,1)$\,, we have \,$p\subset T$\:, and therefore,
\vvn.3>
\be
T+O'\subset\chl\:\bigl(T+\Ne\:(\bW^\Dl_{s_{a\<,a+1}(I),\:\Jt})\:,
T+\Ne\:(\bW^\Dl_{I\<,\:\Jt})\bigr)\,=\,
T+\chl\:\bigl(\Ne\:(\bW^\Dl_{s_{a\<,a+1}(I),\:\Jt})\:,\>
\Ne\:(\bW^\Dl_{I,\Jt})\bigr)\,.
\vv.3>
\ee
By the induction assumption, both \;$\Ne\:(\bW^\Dl_{s_{a\<,a+1}(I),\Jt})$
\vvn.16>
\,and \;$\Ne\:(\bW^\Dl_{I,\Jt})$ \,are contained in \,$O_{\Jt\<,\:\Jt}$\:.
Thus \;$T+O'\?\subset\:T+O_{\Jt\<,\:\Jt}$\,, which by cancellation law yields
\,$O'\?\subset\:O_{\Jt\<,\:\Jt}$\,. Therefore,
\vvn.2>
\be
O_{I\<,\:J}\,=\,K_a(O')\subset K_a(O_{\Jt\<,\:\Jt})\,=\,O_{J,\:J}\,,
\vv.2>
\ee
where the first equality follows from formula \Ref{ee2} and the last equality
is due to Lemma~\ref{lem:extralem}.

\vsk.2>
The case \,$k>l$ \,is proved similarly.
\end{proof}

\begin{example}
For \,$N=2$\,, \,$\la=(1,2)$\,, let
\be
I_1=\:(\{1\},\{2,3\})\,,\quad I_2=\:(\{2\},\{1,3\})\,,\quad
I_3=\:(\{3\},\{1,2\})\,,
\vv.2>
\ee
and let \,$(\nu_1,\nu_2)\in\R^2$ \,be such that
\;$\nu_2\<-\nu_1\<=m_{1,2}\<+\eps_{1,2}$\,, where \,$\eps_{1,2}\<\in(0,1)$\,.
Then we have
\vvn.3>
\begin{align*}
W^\Dl_{I_1}\,&{}=\,(1-h)\>(t/z_1)^{m_{1,2}}\>(1-z_2/t)(1-z_3/t)\,,
\\[4pt]
W^\Dl_{I_2}\,&{}=\,(1-h)\>(t/z_2)^{m_{1,2}}\>(1-hz_1/t)(1-z_3/t)\,,
\\[4pt]
W^\Dl_{I_3}\,&{}=\,(1-h)\>(t/z_3)^{m_{1,2}}\>(1-hz_1/t)(1-hz_2/t)\,.
\\[-23pt]
\end{align*}
Hence
\vvn-.6>
\begin{align*}
O_{I_1,I_1}\,&{}=\,\Ne\:\bigl((1-h)\>(1-z_2/z_1)\>(1-z_3/z_1)\cdot
z_1^{-\nu_1}\:(z_2z_3)^{-\nu_2}\bigr)\,,
\\[4pt]
O_{I_2,I_1}\,&{}=\,\Ne\:\bigl((1-h)^2\>(z_1/z_2)^{m_{1,2}}\>(1-z_3/z_1)\cdot
z_2^{-\nu_1}\:(z_1z_3)^{-\nu_2}\bigr)\,,
\\[4pt]
O_{I_3,I_1}\,&{}=\,\Ne\bigl((1-h)^2\>(z_1/z_3)^{m_{1,2}}\>(1-hz_2/z_1)\cdot
z_3^{-\nu_1}\:(z_1z_2)^{-\nu_2}\bigr)\,.
\end{align*}
Therefore, the statements \,$O_{I_2,\:I_1}\subset O_{I_1,\:I_1}$ and
\,$O_{I_3,\:I_1}\subset O_{I_1,\:I_1}$ after arranging the $\zz$-monomials
to the left\:-hand sides reduce to the statements
\vvn.3>
\be
\Ne\:\bigl((1-h)^2\>(1-z_3/z_1)\>(z_2/z_1)^{\eps_{1,2}}\bigr)\>\subset\,
\Ne\bigl((1-h)\>(1-z_2/z_1)\>(1-z_3/z_1)\bigr)
\vv-.5>
\ee
and
\vvn-.1>
\be
\Ne\bigl((1-h)^2\>(1-hz_2/z_1)\>(z_3/z_1)^{\eps_{1,2}}\bigr)\>\subset\,
\Ne\bigl((1-h)\>(1-z_2/z_1)\>(1-z_3/z_1)\bigr)\,,
\vv.2>
\ee
respectively, see Figure 1.
\end{example}

%%%%%%%%%%%

\begin{figure}
\centering
\begin{tikzpicture}[scale=.9] % easy way to make the picture bigger or smaller

% The four corners of the shaded parallelogram, counterclockwise
\coordinate (A) at (0,0);
\coordinate (B) at (3,0);
\coordinate (C) at (4.5,2.5981);
\coordinate (D) at (1.5,2.5981); % also the ray end 3(cos(pi/3), sin(pi/3))

% the other two ray ends (radius of the circle used is 3)
\coordinate (E) at (-3,0);
\coordinate (F) at (1.5,-2.5981); % = 3(cos(5pi/3),sin(5pi/3))

% the four extra points on the parallelogram, counterclockwise (epsilon=1)
\coordinate (G) at (1,0);
\coordinate (H) at (3.5,.866);
\coordinate (I) at (2.5,2.5981);
\coordinate (J) at (.5,.866);

% draw the rays
\draw [blue, thick] (A) -- (D);
\draw [blue, thick] (A) -- (E);
\draw [blue, thick] (A) -- (F);

% label the rays
\draw (-3.2,0) node [left] {$\dfrac{z_1}{z_2}$};
\draw (1.2,2.6) node [left] {$\dfrac{z_3}{z_1}$};
\draw (2.3,-2.5) node [left] {$\dfrac{z_2}{z_3}$};
\draw (-.05,.4) node [left] {$1$};
\draw (3.9,-.4) node [left] {$\dfrac{z_2}{z_1}$};

% Shade the parallelogram
	\filldraw[dashed, % make the sides of the parallelogram dashed, remove for solid
	draw=gray, % make the shading balboa mist (or another color if desired)
	fill=gray!20, % fill color!darkness
	] (A)
-- (B)
-- (C)
-- (D);

% draw the segments across the parallelogram and give them dots at the ends
\draw [black, thick] (G) -- (I);
\draw [black, thick] (H) -- (J);

\draw [red, thick] (A) -- (B);
\draw [red, thick] (A) -- (D);

\draw[fill=black] (G) circle (2pt) ;
\draw[fill=black] (H) circle (2pt) ;
\draw[fill=black] (I) circle (2pt) ;
\draw[fill=black] (J) circle (2pt) ;

\draw[fill=black] (A) circle (2pt) ;
\draw[fill=black] (D) circle (2pt) ;
\draw[fill=black] (E) circle (2pt) ;
\draw[fill=black] (F) circle (2pt) ;

\draw[fill=black] (B) circle (2pt) ;

% draw and label the epsilon shifts
\draw [-latex, red, thick] (4.5,0) -- (5,.866) node [right] {$\left(\dfrac{z_3}{z_1}\right)^{\varepsilon_{1,2}}$};
\draw [-latex, red, thick] (2,3.4641) -- (3,3.4641) node [above] {$\left(\dfrac{z_2}{z_1}\right)^{\varepsilon_{1,2}}$};
%\draw (7,.5) node [left] {$\left(\dfrac{z_3}{z_1}\right)^{\varepsilon_{12}}$}; % control label placement, if desired
%\draw (3.75,4.25) node [left] {$\left(\dfrac{z_2}{z_1}\right)^{\varepsilon_{12}}$};

\end{tikzpicture}
\caption{$\Ne\left((1-z_3/z_1)(z_2/z_1)^{\eps_{1,2}}\right),\Ne\left((1-hz_2/z_1)(z_3/z_1)^{\eps_{1,2}}\right)$ are contained in $\Ne((1-z_2/z_1)(1-z_3/z_1))$.}
\end{figure}
%%%%%%%%%%%%%%%%%%%%%

\section{Partial flag varieties}
\label{Preliminaries from Geometry}

\subsection{Definitions}
\label{sec Partial flag varieties}

Fix natural numbers $N, n$. Let \,$\bla\in\Z^N_{\geq 0}$, \,$|\bla|=\la_1\lsym+\la_N =n$.
Consider the partial flag variety
\;$\Fla$ parametrizing chains of subspaces
\vvn.2>
\be
0\,=\,F_0\subset F_1\lsym\subset F_N =\,\C^n
\vv.2>
\ee
with \;$\dim F_i/F_{i-1}=\la_i$, \;$i=1\lc N$.
Denote by \,$\tfl$ the cotangent bundle of \;$\Fla$, and let $\pi: \tfl \to
\Fla$ be the projection of the bundle.

Let $I=(I_1\lc I_N)$ be a partition of $\{1\lc n\}$ into disjoint subsets
$I_1\lc I_N$. Denote $\Il$ the set of all partitions $I$ with
$|I_j|=\la_j$, \;$j=1\lc N$.

Let $\epsi_1\lc\epsi_n$ be the standard basis of $\C^n$.
For any $I\?\in\Il$, let $x_I\in\Fla$ be the point corresponding to the
coordinate flag $F_1\lsym\subset F_N$, where $F_i$ \,is the span of the
standard basis vectors \;$\epsi_j\in\C^n$ with \,$j\in I_1\lsym\cup I_i$.
We embed $\Fla$ in $\tfl$ as the zero section and consider the points
$x_I$ as points of $\tfl$.

\subsection{Schubert cells, conormal bundles}
\label{sec:flagvar}
For any $\si\in S_n$, we consider the
coordinate flag in $\C^n$,
\vvn-.4>
\be
V^{\si}\>:\,0\,=\,V_0 \subset V_1\lsym\subset V_n=\,\C^n\>,
\vv.2>
\ee
where \,$V_i$ \>is the span of \,$\epsi_{\si(1)}\lc\epsi_{\si(i)}$\,.
For $I\?\in\Il$ we define the {\it Schubert cell}
\vvn.3>
\beq
\label{OmsI}
\Om_{\si\<,\>I}\>=\,\{\:F\in \Fla\ \,|\
\dim (F_p\cap V^\si_q) = \#\{ i \in I_1\lsym\cup I_p \ |\ \si^{-1}(i)\leq q\} \
\forall p\leq N, \forall q\leq n \}\,.\kern-.6em
\vv.2>
\eeq
The Schubert cell $\Om_{\si\<,\>I}$ is an affine space of dimension
\vvn.2>
\beq
\label{ellsI}
\ell_{\si\<,\>I}\:=\,\#\{\:(a,b)\in\{1\lc n\}^2\ |\ \:
\si(a)\in I_j\>,\,\,\si(b)\in I_k\>,\,\,a>b\,,\,\,j<k\>\}\,.\kern-1em
\vv.1>
\eeq
For a fixed $\si$, the flag manifold is the disjoint union of the cells $\Om_{\si\<,\>I}$. We have $x_I\in \Om_{\si\<,\>I}$,
see e.g. \cite[Sect.2.2]{FP}.

\vsk.2>
For $\si\in S_n$, recall the combinatorial partial ordering \,$\leq_{\:\si}$
on \,$\Il$ \,defined in Section~\ref{sec:e-sub}. It is the same as
the {\it geometric} partial ordering on \,$\Il$\,: \,for \,$I,J\<\in\Il$\,,
we have that \,$J\leq_{\:\si}\<I$ \,if and only if \;$x_J$ \>lies in
the closure of \,$\Om_{\si\<,\>I}$\,.

\vsk.2>
The Schubert cell $\Om_{\si\<,\>I}$ is a smooth submanifold of $\Fla$,
hence we can consider its conormal space
\vvn-.3>
$$
C\Om_{\si\<,\>I}=\{\al\in \pi^{-1}(\Om_{\si\<,\>I}) \ | \
\al(T_{\pi(\al)}\Om_{\si\<,\>I})=0\} \subset\tfl\,,
\vv.2>
$$
cf.~the Remark in Section \ref{sec:edefs}.
The conormal space $C\Om_{\si\<,\>I}$ is the total space of a vector
subbundle of $\tfl$ over $\Om_{\si\<,\>I}$. The rank of this subbundle
is $\dim \Fla - \dim \Om_{\si\<,\>I}$. Hence, as a manifold
$C\Om_{\si\<,\>I}$ is an affine cell of dimension $\dim \Fla$.
In particular, its dimension is independent of $\si, I$.
Define
\vvn.3>
\beq
\label{SLOPE}
\Slope_{\si\<,\>I}\>=\,\tbigcup_{J \leq_{\:\si} I} C\Om_{\si\<,\>J}.
\eeq

\subsection{Torus action}
The diagonal action of the torus $(\C^\times)^n$ on $\C^n$ induces an action
on $\Fla$, and hence on the cotangent bundle $\tfl$.
We extend this $(\C^\times)^n$-action to the action of
$T=(\C^\times)^n\times\C^\times$ so that the extra $\C^\times$ acts
on the fibers of $\tfl\to\Fla$ by multiplication.

The fixed points of $T$ acting on $\tfl$ are the points $x_I$ described above.

\smallskip
In the next sections we consider the equivariant $K\<$-theory and equivariant elliptic cohomology of $\tfl$ with respect to this $T$-action.

\subsection{Equivariant $K\<$-theory of $\tfl$}
\label{sec:equivK}

We consider the equivariant $K\<$-theory algebra
\\
$K_T(\tfl)$.
Our general reference for equivariant $K\<$-theory is \cite[Ch.5]{ChG}.

\vsk.2>
Denote $S_\bla=S_{\la_1}\!\lsym\times S_{\la_N}$ the product of symmetric groups.
Consider variables $\Gm_i=\{\gm_{i,1}\lc\gm_{i,\la_i}\}$, $i=1\lc N$.
Let \;$\GG=(\Gm_1\:\lsym;\Gm_N)$. The group $S_\bla$ acts on the set
$\GG$ by permuting the variables with the same first index.
Let $\C[\GG^{\pm1}]$ be the algebra of Laurent polynomials in variables
$\gm_{i,j}$ and $\CGs$ the subalgebra of
invariants with respect to the $S_\bla$-action.

\vsk.2>
Consider variables $\zz=\{\zzz\}$ and $h$. The group $S_n$ acts
on the set $\zz$ by permutations.
Let $\Czh$ be the algebra of Laurent polynomials in
variables $\zz,h$ and $\Czh^{\>S_n}$ the subalgebra of
invariants with respect to the $S_n$-action. We have
\vvn.2>
\beq
\label{Hrel}
K_T(\tfl)\,=\,\CGs\?\ox\Czh\;\big/\bigl\bra\:f(\GG)=f(\zz)\quad
\on{for\ any} \;f\in\CZS\bigr\ket\,.
\eeq
Here $\gm_{i,j}$ correspond to (virtual) line bundles also denoted
by $\gm_{i,j}$ with
\be
\tbigoplus_{j=1}^{\la_i}\,\gm_{i,j}\>=\>F_i/F_{i-1},
\ee
while $z_a$ and $h$ correspond to the factors of
$T=(\C^\times)^n\times\C^\times$.

\vsk.2>
The algebra \,$K_T(\tfl)$ \>is a module over \,$K_T(\pti\>;\C)=\Czh$.

\subsection{Equivariant localization -- moment map description}
\label{sec:loc}

Consider the {\it equivariant localization} map
\vvn.2>
\beq
\label{eqn:loc}
\Loc\,:\,K_T(\tfl)\,\to\,K_T((\tfl)^T)\,=\,\tbigoplus_{I\in\Il} K_T(x_I)
\eeq
whose components are the restrictions to the fixed points $x_I$. Restriction of the class
$\alpha\in K_T(\tfl)$ to the fix point $x_I$ will be denoted by $\alpha|_{x_I}$. In terms of the variables $\GG$ restriction to $x_I$ is the substitution
\beq
\label{LocI}
\{\gm_{k,1}\lc\<\gm_{k,\la_k}\}\mapsto\{z_a\,|\;a\in I_k\}
\qquad \text{for all}\;\,\,k=1\lc N\:.
\eeq
Equivariant localization theory (see e.g. \cite[Ch.5]{ChG},
\cite[Appendix]{RoKu}) asserts that $\Loc$ is an injection of algebras.
Moreover, an element of the right-hand side is in the image of \;$\Loc$
\,if the difference of the $I$-th and $s_{I\<,\:J}(I)$-th components is divisible
by $1-z_i/z_j$ in $\Czh$ for all $I\?\in\Il$ and $i,j\in \{1\lc n\}$.
Here $s_{I\<,\:J}(I)$ is the partition obtained from $I$ by switching the numbers
$i$ and $j$.

\subsection{Normal Euler classes of conormal spaces at torus fixed points}
\label{sec:edefs}
Define the following classes
\vvn.3>
\beq
\label{ehor}
e_{\si\<,\>I\<,+}^{\hor}=\,\prod_{k<l}\>\prod_{\si(a)\in I_k}\<
\prod_{\satop{\si(b)\in I_l}{b<a}}\!(1-z_{\si(a)}/z_{\si(b)})\,,
\kern1.2em
e_{\si\<,\>I\<,-}^{\hor}=\,\prod_{k<l}\>\prod_{\si(a)\in I_k}\<
\prod_{\satop{\si(b)\in I_l}{b>a}}\!(1-z_{\si(a)}/z_{\si(b)})\,,\kern-.6em
\eeq
\beq
\label{evertr}
e_{\si\<,\>I\<,+}^{\vert}=\,\prod_{k<l}\,\prod_{\si(a)\in I_k}\<
\prod_{\satop{\si(b)\in I_l}{b>a}}\!(1-hz_{\si(b)}/z_{\si(a)})\,,
\kern1.2em
e_{\si\<,\>I\<,-}^{\vert}=\,\prod_{k<l}\,\prod_{\si(a)\in I_k}\<
\prod_{\satop{\si(b)\in I_l}{b<a}}\!(1-hz_{\si(b)}/z_{\si(a)})\,,\kern-.6em
\eeq
in \,$K_T(x_I)=\Czh$\,.
We also set $e_{\si\<,\>I}=\:e_{\si\<,\>I\<,-}^{\hor}\>e_{\si\<,\>I\<,-}^{\vert}$\,.

\vsk.2>
Recall that if \,$\C^\times\!$ acts on a line \,$\C$ \>by
\,$\al\cdot x=\al^rx$,
then the \,$\C^\times\<$-equivariant Euler class of the line bundle
$\C \to \{0\}$ is $e(\C\to \{0\})=1-z^r \in K_{\C^\times}($point$)=\C[z^{\pm1}]$.
Thus standard knowledge on the tangent bundle of flag manifolds imply that
\vvn.2>
\beq\label{eqn:einter1}
e(T\Om_{\si\<,\>I}|_{x_I})=\,e_{\si\<,\>I\<,+}^{\hor}\,, \qquad
e(\nu(\Om_{\si\<,\>I}\subset\Fla)|_{x_I})=\,e_{\si\<,\>I\<,-}^{\hor}\,,
\vv.2>
\eeq
where \,$\nu(A\subset B)$ means the normal bundle of a submanifold $A$
in the ambient manifold $B$, and $\xi|_x$ means the restriction of
the bundle $\xi$ over the point $x$ in the base space.
Therefore we also have
\vvn-.2>
\be
e(C\Om_{\si\<,\>I}|_{x_I})\>=\>e_{\si\<,\>I\<,+}^{\vert}\,,\qquad
e((\pi^{-1}(\Om_{\si\<,\>I})-C\Om_{\si\<,\>I})|_{x_I})\>=\>e_{\si\<,\>I\<,-}^{\vert}\,,
\vv.2>
\ee
where $C\Om_{\si\<,\>I}$ and $\pi^{-1}(\Om_{\si\<,\>I})$ are
considered bundles over $\Om_{\si\<,\>I}$.
Now consider \,$C\Om_{\si\<,\>I}$ \,as a(n open) submanifold of \,$\tfl$\>.
Then we obtain
\vvn.2>
\beq\label{eqn:enu}
e(\nu( C\Om_{\si\<,\>I} \subset\tfl)|_{x_I})\,=
\,e^{\hor}_{\si,I,-}\>e^{\vert}_{\si,I,-}\>=\,e_{\si\<,\>I}\,.
\vv.2>
\eeq

\begin{rem}
For a permutation \,$\si\in S_n$\,, let the cocharacter
$\chi:\C^\times\to(\C^\times)^n$, \,$t\mapsto (t^{w_1},\ldots,t^{w_n})$
satisfies the property: $w_{\sigma(a)}>w_{\sigma(b)}$ for all $a>b$.
The formulas of this subsection show that
$C\Om_{\si\<,\>I}=\{x\in \tfl\ |\ \lim_{t\to 0} \chi(t)x=x_I\}$.
This latter set is a special case for $\tfl$ of the notion ``stable leaf'' of the fixed point $x_I$
corresponding to the ``chamber'' $\{(w_1,\ldots,w_n)\in \R^n\ |\ w_{\si(a)}>w_{\si(b)}$ for all
$a>b)\}\subset\on{Cochar}\bigl((\C^\times)^n\bigr)\otimes_\Z\R
=\R^n$ in \cite[Section 3.2.2]{MO1}. Thus the notion of ``chambers'' of
\cite{MO1} correspond to the permutations $\si\in S_n$ in this paper.
\end{rem}

\section{$K\<$-theoretic stable envelopes}
\label{sec:5}

\subsection{Definition of $K\<$-theoretic stable envelopes}

The $K\<$-theoretic stable envelope maps for Nakajima varieties $X$ are defined by Maulik and Okounkov in \cite{MO2}. That definition is reproduced in \cite[Section 9.1]{O} and \cite[Section 2.1]{OS}. Here we recall their axiomatic definition of stable envelope classes (images of coordinate vectors under the stable envelope map) in the special case $X=\tfl$ acted upon by $(\C^*)^n\times \C^*$.

Let $\Dl\subset \R^N$ be an alcove and $\nu\in\Dl$.
Define the
virtual ``slope'' bundle
\be
S_\nu=\prod_{k=1}^N \prod_{a=1}^{\la_k} \gamma_{k,i}^{-\nu_k} \ \in \ \Pic (\tfl)\ox_\Z \R.
\ee
Observe that $S_\nu$ restricted to the fixed point $x_I$ is $S_I$ from Section \ref{sec:Newton}.

\begin{defn} (\cite{MO2}, \cite[Section 9.1]{O}, \cite[Section 2.1]{OS}) \label{def:o}
An element $\Stab^\Dl_{\si\<,\>I}\in K_T(\tfl)$ for an alcove $\Dl\subset \R^N$, $\si\in S_n$, $I\in\Il$
is called the {\it stable envelope class}, if it satisfies the following axioms:
\renewcommand{\theenumi}{\Roman{enumi}}
\begin{enumerate}
\item{} \label{o1} $\Stab^\Dl_{\si\<,\>I}$ is supported on $\Slope_{\si\<,\>I}$;
\item{} \label{o2} $\Stab^\Dl_{\si\<,\>I}|_{x_I}= h^{\codim(\Om_{\si\<,\>I}\subset \Fla)/2} P_{\si\<,\>I}e_{\si\<,\>I}$;
\item{} \label{o3}
$
\Ne\:( \Stab^\Dl_{\si\<,\>I}|_{x_J} \cdot S_\nu|_{x_I} ) \subset
\Ne\:( \Stab^\Dl_{\si\<,\>J}|_{x_J} \cdot S_\nu|_{x_J} ),
$
if $\nu \in \Dl$.
\end{enumerate}
\end{defn}

Observe that while condition (\ref{o3}) depends on $\nu \in \R^N$, the class $\Stab^\Dl_{\si\<,\>I}$ only depends on the alcove $\nu$ belongs to. Namely, the inclusion in (\ref{o3}) can be reformulated as
\be
\Ne\:( \Stab^\Dl_{\si\<,\>I}|_{x_J} ) +
\Ne\:(\frac{ S_\nu|_{x_I}}{S_\nu|_{x_J}})
\subset
\Ne\:( \Stab^\Dl_{\si\<,\>J}|_{x_J} ),
\ee
where $+$ is the Minkowski sum and $\Ne\:(\frac{ S_\nu|_{x_I}}{S_\nu|_{x_J}})$ is a point. Thus
\Ref{o3} says that the Newton polytope of $ \Stab^\Dl_{\si\<,\>I}|_{x_J}$ lies inside the Newton polytope of $\Stab^\Dl_{\si\<,\>J}|_{x_J} $ when shifted by the vector $\Ne\:(\frac{ S_\nu|_{x_I}}{S_\nu|_{x_J}})$ which depends on $\nu$ in $\Dl$.

\begin{rem} In fact, in \cite{MO2,O,OS} the definition of stable envelope classes also depends on a choice of distinguishing half of the normal directions at each torus fixed point. For $\tfl$ there is a natural choice -- distinguishing the cotangent directions -- that we make throughout the paper.
Moreover, condition (\ref{o2}) in the cited papers reads as follows
\beq\label{eqn:omain}
\Stab^\Dl_{\si\<,\>I}|_{x_I}=
(-1)^{ \dim (\text{distinguished directions in $C\Om_{\si\<,\>I}$})}
\sqrt{ \frac{ \det( \nu_{\si\<,\>I} )}{\det( \tfl|_{x_I})} }
e(\nu_{\si\<,\>I}),
\eeq
where $\nu_{\si\<,\>I}$ is the normal bundle of $C\Om_{\si\<,\>I}$ at $x_I$. However, using the notation
\be
d_{\si,I,+}^{\hor}=
\prod_{k<l}\prod_{\si(a)\in I_k} \prod_{\satop{\si(b)\in I_l}{b<a}} \frac{z_{\si(b)}}{z_{\si(a)}},
\qquad
d_{\si,I,-}^{\hor}=
\prod_{k<l}\prod_{\si(a)\in I_k} \prod_{\satop{\si(b)\in I_l}{b>a}} \frac{z_{\si(b)}}{z_{\si(a)}},
\ee
\be
d_{\si,I,+}^{\vert}=
\prod_{k<l}\prod_{\si(a)\in I_k} \prod_{\satop{\si(b)\in I_l}{b>a}} \frac{z_{\si(a)}}{hz_{\si(b)}},
\qquad
d_{\si,I,-}^{\vert}=
\prod_{k<l}\prod_{\si(a)\in I_k} \prod_{\satop{\si(b)\in I_l}{b<a}} \frac{z_{\si(a)}}{hz_{\si(b)}}
\ee
the right hand side of equation (\ref{eqn:omain}) is
\be
\prod_{k<l}\prod_{\si(a)\in I_k} \prod_{\satop{\si(b)\in I_l}{b>a}} (-1)
\cdot
\sqrt{
\frac{
d^{\hor}_{\si,I,-} d^{\vert}_{\si,I,-}
}{
d^{\vert}_{\si,I,+} d^{\vert}_{\si,I,-}
}
}
\cdot e_{\si\<,\>I}=
\ee
\be
\prod_{k<l}\prod_{\si(a)\in I_k} \prod_{\satop{\si(b)\in I_l}{b>a}} \left(-\frac{z_{\si(b)}}{z_{\si(a)}}\right)
h^{\codim(\Om_{\si\<,\>I}\subset \Fla)/2} e_{\si\<,\>I}=h^{\codim(\Om_{\si\<,\>I}\subset \Fla)/2} P_{\si\<,\>I} e_{\si\<,\>I},
\ee
as stated in the definition above.
\end{rem}

It is proved in \cite[Theorem 9.2.1]{O} that stable envelope classes -- if exist -- are unique. Their existence statement for Nakajima varieties is announced to be published in \cite{MO2}, see \cite[9.1.1]{O}, cf.~also Theorem 2 and Proposition 3.5 in \cite{AO}.

\subsection{Weight functions project to stable envelope classes}

We substitute the variables
\be
t^{(k)}_1,\ldots,t^{(k)}_{\la^{(k)}}
\quad\mapsto\quad
\gamma_{1,1},\ldots,\gamma_{1,\la_{1}},\ \
\gamma_{2,1},\ldots,\gamma_{2,\la_{2}},\ \
\ldots,\ \
\gamma_{k,1},\ldots,\gamma_{k,\la_{k}}
\ee
into the function $\Wt^\Dl_{\si\<,\>I}(\ttt,\zz,h)$ (see Section \ref{sec:trigweight}). The resulting function will be denoted by $\Wt^\Dl_{\si\<,\>I}(\GG,\zz,h)$. Observe that we have
\be
\Wt^\Dl_{\si\<,\>I}(\GG,\zz,h)|_{x_J}
=
\Wt^\Dl_{\si\<,\>I}(z_J,\zz,h).
\ee
In view of Lemma \ref{lem:WK divisible by E} this implies that the function $\Wt^\Dl_{\si\<,\>I}(\GG,\zz,h)$ -- although not a Laurent polynomial itself -- has the property that all its ${x_J}$ restrictions are Laurent polynomials. The restrictions also satisfy the divisibility property of Section \ref{sec:loc}, hence equivariant localization theory implies that there is a unique element in $K_T(\tfl)$ whose $x_J$ restrictions are the same as those of $\Wt^\Dl_{\si\<,\>I}(\GG,\zz,h)$. By slight abuse of language we will denote this element of $K_T(\tfl)$ by $[\Wt^\Dl_{\si\<,\>I}(\GG,\zz,h)]$.

\begin{thm}\label{thm:OisW}
Stable envelope classes $\Stab^\Dl_{\si\<,\>I}$ exist for $\tfl$, namely,
\beq\label{eqn:StabIsW}
\Stab^\Dl_{\si\<,\>I}=
h^{\codim(\Om_{\si\<,\>I}\subset \Fla)/2}
\cdot
[\Wt^\Dl_{\si\<,\>I}(\GG,\zz,h)].
\eeq
\end{thm}

\begin{proof}
We denote the right hand side of (\ref{eqn:StabIsW}) by $\om_{\si\<,\>I}$ and we will prove that it satisfies axioms (\ref{o1})-(\ref{o3}).

Axiom (\ref{o2}) follows from Lemma \ref{lem:expected}. To prove axiom (\ref{o3}) observe that
$S_\nu|_{x_I}=S_I$ and hence Theorem \ref{thm:Newton} claims
\be
\Ne\:(W_I^\Dl(\zz_J,\zz,h)S_\nu|_{x_I}) \subset \Ne\:(W_J^\Dl(\zz_J,\zz,h)S_\nu|_{x_J}).
\ee
Writing this statement for the pair $\si^{-1}(I)$, $\si^{-1}(J)$ instead of $I, J$, we obtain
\be
\Ne\:(W_{\si^{-1}(I)}^\Dl(\zz_{\si^{-1}(J)},\zz,h)S_\nu|_{x_{\si^{-1}(I)}})
\subset
\Ne\:(W_{\si^{-1}(J)}^\Dl(\zz_{\si^{-1}(J)},\zz,h)S_\nu|_{x_{\si^{-1}(J)}}).
\ee
Replace $z_a$ with $z_{\si(a)}$ for $a=1,\ldots,n$, to get
\be
\Ne\:(W_{\si^{-1}(I)}^\Dl(\zz_{J},\zz_\si,h)S_\nu|_{x_I})
\subset
\Ne\:(W_{\si^{-1}(J)}^\Dl(\zz_{J},\zz_\si,h)S_\nu|_{x_{J}}),
\ee
where $\zz_\si=(z_{\si(1)},\ldots,z_{\si(n)})$. Using the definition of $W_{\si\<,\>I}$ functions this translates to
\be
\Ne\:(W_{\si\<,\>I}^\Dl(\zz_{J},\zz,h)S_\nu|_{x_I})
\subset
\Ne\:(W_{\si\<,\>J}^\Dl(\zz_{J},\zz,h)S_\nu|_{x_{J}}).
\ee
Since the ratio between $W_I^\Dl(\zz_J,\zz,h)$ and $\Wt_I^\Dl(\zz_J,\zz,h)$ is a universal expression (not depending on $I$) the cancellation property of Minkowski sums gives
\be
\Ne\:(\Wt_{\si\<,\>I}^\Dl(\zz_J,\zz,h)S_\nu|_{x_I})
\subset
\Ne\:(\Wt^\Dl_{\si\<,\>J}(\zz_J,\zz,h)S_\nu|_{x_J}),
\ee
what we wanted to prove.

The rest of this section is the proof of axiom (\ref{o1}).
In the proof we will consider restrictions of elements of $K_T(\tfl)$ to $x_J$,
to $\Om_{\si\<,\>J}$, and to $\pi^{-1}(\Om_{\si\<,\>J})$. Observe that
algebraically these three restrictions are the same maps, because the last two
sets are topological cells, equivariantly homotopy equivalent to the one-point
set $x_J$.

\begin{lem} \label{lem:econd}
Let $\prec$ be a linear order refining $\leq_{\:\si}$, and let $J^{(1)}\prec J^{(2)} \prec \ldots \prec
J^{(d)}$ be the elements of $\Il$, that is $d=|\Il|$. Let $a\in\{2,\ldots,d\}$. If a class $\om \in K_T(\tfl)$ restricted to
\be
\cup_{i=a}^d \pi^{-1}( \Om_{\si\<,\>J^{(i)}})
\ee
is 0, and $\om|_{x_{J^{(a-1)}}}=0$, then $\om$ restricted to
\be
\cup_{i=a-1}^d \pi^{-1}( \Om_{\si\<,\>J^{(i)}})
\ee
is also 0.
\end{lem}

\begin{proof} This is a well known argument involving a Mayer-Vietoris and a Gysin sequence, and depending on the non-vanishing of a normal Euler class (which holds for partial flag varieties). See for example \cite[Lemma 3.6]{FR} for a detailed argument in cohomology.
\end{proof}

\begin{lem} \label{lem:gysin}
Let $T$ act on the vector space $X$ with an invariant subspace $Y$. Let $e$ be the $T$-equivariant normal Euler class of $Y\subset X$. Then the following conditions are equivalent for a class in $K_T(X)$:
\begin{itemize}
\item it is supported on $Y$;
\item it is divisible by $e$.
\end{itemize}
\end{lem}

\begin{proof}
The statement follows from the exactness of the Gysin sequence
\be
\xymatrix{
K_T(Y)\ar[r]^{\cdot e} &
K_T(Y)\ar[r] &
K_T(X-Y)
}
\ee
where the second map is the composition of the isomorphism $K_T(Y)=K_T(X)$ and the restriction to $X-Y$.
\end{proof}

\begin{lem} \label{lem:LEO}
We have
\be
\om_{\si\<,\>I}=\sum_{J\leq_{\:\si} I} f_{IJ}(\zz,h)
[\>\overline{\<C\Om_{\si\<,\>J}\<}\>]
\ee
for some Laurent polynomials $f_{IJ}(\zz,h)$.
\end{lem}

\begin{proof}
Let $\prec$ be a total order refining $\leq_{\:\si}$, and let
$
J^{(1)} \prec J^{(2)} \prec \ldots \prec J^{(r)}=I
$
be the elements of $\Il$ that are $\leq_{\si} I$. Let $k\in\{0,\ldots, r\}$. We will prove the following statement by induction on $k$.

\smallskip
\noindent{\em (*k) There exist Laurent polynomials $f_{IJ^{(i)}}(\zz,h)$ for $i\in \{r-k+1,\ldots,r\}$ such that
\beq\label{eqn:ind}
\om_{\si\<,\>I}-
\sum_{i=r-k+1}^r f_{IJ^{(i)}}(\zz,h)[\>\overline{\<C\Om_{\si\<,\>J^{(i)}}\<}\>]
\eeq
is supported on $\cup_{i=1}^{r-k} \pi^{-1}(\Om_{\si\<,\>J^{(i)}})$.}
\smallskip

Statement (*k) for $k=0$ follows from the iterated application of Lemma~\ref{lem:econd}, using Lemma~\ref{lem:triang}.

For the induction step $k \mapsto k+1$, consider (\ref{eqn:ind}) restricted
to $x_{J^{(r-k)}}$. The term $\om_{\si\<,\>I}|_{x_{J^{(r-k)}}}$ is divisible
by $e^{\vert}_{\si\<,\>J^{(r-k)},-}$ by Lemma~\ref{lem:div_e_ver}. The classes
$[\>\overline{\<C\Om_{\si\<,\>J^{(i)}}\<}\>]$ appearing in (\ref{eqn:ind}),
restricted to $x_{J^{(r-k)}}$ are also divisible by
$e^{\vert}_{\si\<,\>J^{(r-k)},-}$, due to the fact that the Schubert cell
stratification is a Whitney stratification. Therefore (\ref{eqn:ind}),
restricted to $x_{J^{(r-k)}}$ is divisible by
$e^{\vert}_{\si\<,\>J^{(r-k)},-}$.

We claim that (\ref{eqn:ind}), restricted to $x_{J^{(r-k)}}$ is also divisible by $e^{\hor}_{\si\<,\>J^{(r-k)},-}$. This follows from the induction hypotheses and Lemma~\ref{lem:gysin}.
We obtained that (\ref{eqn:ind}) restricted to $x_{J^{(r-k)}}$ can be written as
$f_{IJ^{(r-k)}}(\zz,h) \cdot e^{\hor}_{\si\<,\>J^{(r-k)},-} e^{\vert}_{\si\<,\>J^{(r-k)},-}$ for some Laurent polynomial $f_{IJ^{(r-k})}(\zz,h)$. Therefore the class
\be
\left(
\om_{\si\<,\>I}-
\sum_{i=r-k-1}^r f_{IJ^{(i)}}(\zz,h)[\>\overline{\<C\Om_{\si\<,\>J^{(i)}}\<}\>]
\right)
-
f_{IJ^{(r-k)}}(\zz,h) [\>\overline{\<C\Om_{\si\<,\>J^{(r-k)}}\<}\>]
\ee
restricted to $x_{J^{(r-k)}}$ is 0. Since it is also 0 when restricted to $\cup_{i=r-k-1}^{r} \pi^{-1}(\Om_{\si\<,\>J^{(i)}})$ (by induction and the property of the order $\prec$), Lemma \ref{lem:gysin} gives the induction step.

Statement (*k) for $k=r$ proves Lemma \ref{lem:LEO}.
\end{proof}

The class $[\>\overline{\<C\Om_{\si\<,\>I}\<}\>]$ is supported on
$\Slope_{\si\<,\>I}$, and clearly $J\leq_{\:\si} I$ implies
$\Slope_{\si\<,\>J}\subset \Slope_{\si\<,\>I}$, hence Lemma~\ref{lem:LEO}
proves that $\om_{\si\<,\>I}$ satisfies axiom (\ref{o3}). This finishes the
proof of Theorem~\ref{thm:OisW}.
\end{proof}

\begin{rem}
Our arguments show that in the definition of $\Stab_{\si\<,\>I}$ axiom (\ref{o1}) can be replaced with its local version

(I') $\Stab^\Dl_{\si\<,\>I}|_{x_J}$ is divisible by $e^{\vert}_{\si\<,\>J,-}$.

\noindent as was done, for example, in \cite[Theorem 3.1]{RTV2}.
\end{rem}

\begin{rem}
The functions $f_{IJ}(\zz,h)$ appearing in Lemma~\ref{lem:LEO} are related to the notion of ``local Euler obstruction'' in cohomology.
\end{rem}

\subsection{Note on $R$-matrices}
\label{NoR}

A key feature of the $\Stab^\Dl_{\si\<,\>I}$ classes is that some endomorphisms
encoded by them satisfy Yang-Baxter equations, that is, they are $R$-matrices.
For more details, see eg. \cite{O, RTV2}. For completeness here we show the
prototype $R$-matrices for the convention used in this paper.

Let $\XX_n=\cup_{|\la|=n} \tfl$. For $I\in \Il$ let $v_I=v_{i_1}\otimes\ldots \otimes v_{i_n}\in (\C^N)^{\otimes n}$ for $i_j=k$ if $j\in I_k$. Stable envelope maps are defined
\be
\St_{\si}^\Delta: (\C^N)^{\otimes n}\otimes \C(\zz,h)=K_T\left( \XX_n^T \right)\otimes \C(\zz,h)
\to
K_T\left( \XX_n\right)\otimes \C(\zz,h)
\ee
by
$
v_I \mapsto \Stab^\Dl_{\si\<,\>I}.
$

Stable envelope maps define ``geometric'' $R$-matrices by
\be
\Rc_{\si',\si}^\Delta=(\St_{\si'}^\Delta)^{-1}\circ \St_{\si}^\Delta.
\ee

For example, for $n=2, N=2$, we have
\begin{align*}
\St_{\id}^\Delta(v_1 \ox v_2)=&(\gm/z_1)^{m_{1,2}}\left( 1-z_2/\gm\right)\cdot \sqrt{h}
\\
\St_{\id}^\Delta(v_2 \ox v_1)=&(\gm/z_2)^{m_{1,2}}\left( 1-hz_1/\gm\right)
\\
\St_s^\Delta(v_1 \ox v_2)=&(\gm/z_1)^{m_{1,2}}\left( 1-hz_2/\gm\right)
\\
\St_s^\Delta(v_2 \ox v_1)=&(\gm/z_2)^{m_{1,2}}\left( 1-z_1/\gm\right)\cdot \sqrt{h}.
\end{align*}
Therefore, $\Rc_{s,\id}^\Delta$ acting on $(\C^2)^{\otimes 2}$ with ordered basis $v_1\otimes v_1, v_1\otimes v_2, v_2\otimes v_1, v_2\otimes v_2$ is
\be
\label{RR2o}
\left(1\right)\oplus
\left( \begin{array}{cccc}
\dfrac{\sqrt{h}(1-z_2/z_1)}{1-hz_2/z_1}
& \dfrac{(1-h)(z_2/z_1)^{-m_{1,2}}}{1-hz_2/z_1}
\\[9pt]
\dfrac{(1-h)(z_2/z_1)^{m_{1,2}+1}}{1-hz_2/z_1} & \dfrac{\sqrt{h}(1-z_2/z_1)}{1-hz_2/z_1}
\end{array} \right)\oplus \left(1\right).
\vv.2>
\ee
The $R$-matrix $R_{s,\id}$ used in \cite{RTV2}
\begin{equation*}
\label{RR2}
\left(1\right)\oplus
\left( \begin{array}{cccc}
\dfrac{1-z_2/z_1}{1-hz_2/z_1}
& \dfrac{(1-h)}{1-hz_2/z_1}
\\[9pt]
\dfrac{(1-h)z_2/z_1}{1-hz_2/z_1} & \dfrac{h(1-z_2/z_1)}{1-hz_2/z_1}
\end{array} \right)\oplus \left(1\right)
\end{equation*}
is obtained from this one by putting $m_{1,2}=0$ (considering the anti-dominant alcove) and a
change in the normalization convention involving $h$.

Similar calculation shows that for general $N$, but $n=2$, the operator
$\Rc=\Rc_{s,id}^\Delta$ acting on $(\C^N)^{\otimes 2}$ has entries
\beq
\label{RN}
\Rc_{jj}^{jj}=1, \qquad \Rc^{jk}_{jk}=\sqrt{h}\frac{1-z_2/z_1}{1-hz_2/z_1} \qquad (1\leq j,k \leq N, j\ne k),
\eeq
\be
\Rc^{jk}_{kj}=\frac{(1-h)(z_2/z_1)^{m_{j,k}+1}}{1-hz_2/z_1},
\qquad
\Rc^{kj}_{jk}=\frac{(1-h)(z_2/z_1)^{-m_{j,k}}}{1-hz_2/z_1} \qquad{(1\leq j<k \leq N),}
\ee
and otherwise \,$0$\,, while the $R$-matrix $R=R_{s,\id}$ used in \cite{RTV2} has entries
$R_{jj}^{jj}=1$,
\beq
\label{RO}
R^{jk}_{jk}=\frac{1-z_2/z_1}{1-hz_2/z_1} \quad (j<k ),
\qquad
R^{jk}_{jk}=h\frac{1-z_2/z_1}{1-hz_2/z_1} \quad (k<j),
\eeq
\be
R^{jk}_{kj}=\frac{(1-h)(z_2/z_1)}{1-hz_2/z_1},
\qquad
R^{kj}_{jk}=\frac{1-h}{1-hz_2/z_1} \qquad{(1\leq j<k \leq N),}
\ee
and otherwise \,$0$\,.

\section{Line bundles and quadratic forms}
\label{sec:6}
\subsection{Line bundles on $\C^p$}

Let $E=\C/\La$ be an elliptic curve, where $\La=\Z+\tau\Z$, $\tau \in\C$, $\text{Im}\,\tau>0$,
so that $E^p=\C^p/\La^p$.
For each pair $(M,v)$ consisting of a symmetric integral $p\times p$ matrix $M$
and $v\in (\R/\Z)^p$ let $\LL(M,v)$ be the line bundle
$(\C^p\times \C)/\Lambda^p \to \C^p/\Lambda^p$ with action
\be
\la \cdot (x,u)=(x+\la, e_{\la}(x)(u)),
\qquad
\la\in \Lambda^p,\, x\in \C^p, \, u\in \C,
\ee
and cocycle
\be
e_{n+m\tau}(x)=(-1)^{n^tMn}(-e^{i\pi\tau})^{m^tMm} e^{2\pi i m^t(Mx+v)},
\qquad
n,m\in \Z^p.
\ee

To an integral symmetric $p \times p$ matrix $M$ and a vector $v\in \C^p$ we associate
the integral quadratic form $M(x) = x^t M x$ and the linear form $v(x) = x^tv$
on the universal cover $\C^p$ of $E^p$, and we call them the quadratic form and the linear
form of the line bundle $\LL(M, v)$. The linear form is defined up to addition of an
integral linear form.

\begin{prop}
[{\cite[Proposition~5.1]{FRV}, see \cite[Section~I.2]{Mu}}]

\label{p-li} \
\begin{itemize}
\item[(i)] $\LL(M,v)$ is isomorphic to $\LL(M',v')$ if
and only if $M=M'$ and $v\equiv v'\mod\Lambda^p$.
\item[(ii)] For generic $E$, every holomorphic line bundle on $E^p$
is isomorphic to $\LL(M,v)$ for some $(M,v)$.
\item[(iii)]
$\LL(M_1,v_1)\otimes \LL(M_2,v_2)\cong \mathcal
L(M_1+M_2,v_1+v_2)$.
\item[(iv)] Let $\si\in S_p$ act by permutations on $E^p$ and
$\C^p$. Denote also by $\si$ the corresponding
$p\times p$ permutation matrix. Then
\be
\si^*\LL(M,v)=\LL(\si^t M\si,\si^t v).
\ee
\end{itemize}
\end{prop}

\begin{rem}
\label{rem5.4} Sections of $\LL(M,v)$ are the same as functions
$f$ on $\C^p$ such that $f(x+\la)=e_\la(x)^{-1}f(x)$
for all $\la\in\Lambda^p$, $x\in\C^p$. Explicitly, a
function on $\C^p$ defines a section of $\LL(M,v)$ if
and only if
\begin{align*} f(x_1,\dots,x_j+1,\dots,x_p) &= (-1)^{M_{jj}}f(x), \\
f(x_1,\dots,x_j+\tau,\dots,x_p) &= (-1)^{M_{jj}}e^{-2\pi i(\sum_k
M_{jk}x_k+v_j)-\pi i \tau
M_{jj}} f(x),
\end{align*} for all $x\in\C^p$, $j=1,\dots,p$.
\end{rem}

\begin{rem} [\cite{FRV}] Let $\theta(x)$ be the theta
function in one variable defined by \Ref{TH}. Then, for any $r\in\mathbb Z^p$ and
$z\in\C$,
\be
\theta(r^tx+z)=\theta(r_1x_1+\cdots+r_px_p+z)
\ee
is a holomorphic section of $\LL(M,v)$ with quadratic form
$ M(x)=\left(\sum_{i=1}^pr_ix_i\right)^2$,
and linear form
$ v(x)=z\sum_{i=1}^p r_ix_i$.
Since an
integral quadratic form is an integral linear combination of squares
of integral linear forms, a line bundle $\LL(M,0)$ has a meromorphic
section which is a ratio of products of theta functions
$\theta(r^tx)$ with $r\in\mathbb Z^p$.
\end{rem}

\subsection{Theta functions as sections of line bundles}

Introduce new variables
$\bs v=(v^{(1)}$, \dots, $v^{(N)})$,
$v^{(k)}=(v_{k,1},\dots,v_{k,\la_k})$, $y,$ \,$ \nub=(\nu_1,\dots,\nu_{N-1})$.
Make the substitution
\bean
\label{sub1}
&&
\phantom{aaa}
q,\ h,\ \mu_1/\mu_2,\dots,\mu_{N-1}/\mu_N
\quad \mapsto \quad e^{2\pi i \tau}, \ e^{2\pi i y}, \ e^{2\pi i\nu_1},\ldots,
e^{2\pi i\nu_{N-1}},
\\
\notag
&&
t^{(k)}_1,\ldots,t^{(k)}_{\la^{(k)}}
\
\mapsto
\
e^{2\pi i v_{1,1}},\ldots,e^{2\pi iv_{1,\la_{1}}},\ \
e^{2\pi iv_{2,1}},\ldots,e^{2\pi iv_{2,\la_{2}}},\ \
\ldots,\ \
e^{2\pi iv_{k,1}},\ldots,e^{2\pi i v_{k,\la_{k}}}
\eean
for $k=1,\dots,N$ in the functions $G_\bla(\ttt,h,\mub)$, $E^{\>\on{ell}}_\bla(\ttt,h)$. Denote the resulting functions by
$\hat G_\bla(\bs v,y,\nub)$, $\Eh^{\>\on{ell}}_\bla(\bs v, y)$.

\begin{lem}
\label{lem:sl}
The function $\hat G_\bla(\bs v,y,\nub)/\Eh^{\>\on{ell}}_\bla(\bs v, y)$
defines a meromorphic section of the line bundle $\mc L (M_\bla,0)$ on
$E^{\la_1}\times \ldots \times E^{\la_N}\times E\times E^{N-1}$, where
\bean
\label{mL}
\phantom{aaa}
M_{\bla}(\bs v, y, \nub)
&=&
\sum_{1\leq j<k\leq N} \sum_{a=1}^{\la_j}\sum_{b=1}^{\la_k}(v_{j,a}-v_{k,b})^2
- y^2 \sum_{k=1}^{N-1} (\la_1+\dots+\la_k)^2
\\
\notag
&+&
\sum_{k=1}^{N-1}(\nu_k-\la_ky) \sum_{j=1}^k \sum_{a=1}^{\la_j} v_{j,a} .
\eean

\end{lem}

\begin{proof} By straightforward calculation.
\end{proof}

Make the substitution
\bean
\label{sub2}
&&
q,\ h \quad \mapsto \quad e^{2\pi i \tau}, \ e^{2\pi i y},
\\
\notag
&&
z_1,\dots,z_n, \ \mu_1/\mu_2,\dots,\mu_{N-1}/\mu_N
\quad
\mapsto
\quad
e^{2\pi i x_1},\ldots,e^{2\pi ix_n},\
e^{2\pi i\nu_1},\ldots,
e^{2\pi i\nu_{N-1}},
\eean
in the function $ G_I(\zz,h,\mub)$. The resulting function denote by $\hat G_I(\xx,y,\nub)$.

\begin{lem}
\label{lem:slI}
For $I\in\Il$ the function $\hat G_I(\xx,y,\nub)$
defines a meromorphic section of the line bundle $\mc L (M_I,0)$ on
$E^n\times E\times E^{N-1}$, where
\bean
\label{mI}
\phantom{aaa}
M_I(\xx, y, \nub)
&=&
-2\sum_{k=1}^{N-1}(\nu_k+\dots+\nu_{N-1})\sum_{a\in I_k}x_a
\\
\notag
&&
-\,y^2\,\sum_{k=1}^{N-1}\sum_{a\in I_k}
(p_{I\<\<,\:k}(a)-a+1+\la_1+\dots+\la_{k-1})^2
\\
\notag
&&
+\,2y\,\sum_{k=1}^{N-1}\sum_{a\in I_k} x_a (p_{I\<\<,\:k}(a)-a+1+n-\la_N).
\eean
\end{lem}

\begin{proof} By straightforward calculation.
\end{proof}

\section{Equivariant elliptic cohomology of $\tfl$}
\label{sec7}

\subsection{Tautological bundles and Chern classes}
\label{Cc}

We follow the exposition of equivariant elliptic cohomology in \cite[Section 4]{FRV}, which is based on \cite{Groj,GKV}.

Let $E=\C/(\Z+\tau\Z)$, $\text{Im}\,\tau>0$, be an elliptic curve and $G$ a compact group.
The $G$-equivariant elliptic cohomology $E_G(-)$ is a covariant functor from finite $G$-CW
complexes to superschemes satisfying a set of axioms. For example,
\be
E_{U(1)}(\pt)=E,
\qquad
E_{U(n)}(\pt)=E^{(n)}=E^n/S_n .
\ee
For a $G$-space $M$ we have the structure map
$p_G:E_G(M)\to E_G(\pt)$. The $A=U(1)^n$-equivariant elliptic cohomology of $\Fla$ is the fiber product obtained from the Cartesian square
\be
\xymatrix{
E_A(\Fla) \ar[r]^{\chi\ \hskip 1.5 true cm}\ar[d]_{p_A} &
E^{(\la_1)} \times E^{(\la_2)} \times \ldots \times E^{(\la_N)} \ar[d]^{p_\la} \\
E^n\ar[r]_{p_n} & E^{(n)}
}
\ee
where the left vertical arrow is the structure map to $E_A(\mathrm{pt})$;
$p_n$, $p_\la$ are canonical projections, and the top horizontal map $\chi$ is the
characteristic map coming from the $N$ equivariant bundles $F_{k}/F_{k-1}$ over $\Fla$, see \cite{GKV},
and we have
\beq
\label{dsch}
E_A(\Fla)\,=\,
E^n \times_{E^{(n)}} (E^{(\la_1)} \times E^{(\la_2)} \times \ldots \times E^{(\la_N)})
\eeq

\subsection{Equivariant localization -- moment map description}
The inclusion of the fixed point $x_I$ to $\Fla$ induces a map $\iota_I:E_A(\pt)\to E_A(\Fla)$ and $E_A(\Fla)$ is the union of $\iota_IE_A(\pt)\cong E^n$ where $I\in \Il$. The intersections of $\iota_IE_A(\pt)\cong E^n$'s are described as follows. Assume $a\in I_k$, $b\in I_l$ for an $I\in \Il$, and assume $k\ne l$. Let $s_{a,b}(I)$ be obtained by switching the numbers $a$ and $b$, i.e. $a\in s_{a,b}(I)_l$ and $b\in s_{a,b}(I)_k$. The ``diagonal'' $\Dl_{I,s_{a,b}(I)}=\{\xx\in E^n : x_a=x_b\}$ is naturally included in both
$\iota_IE_A(\pt)=E^n$ and $\iota_{s_{a,b}(I)}E_A(\pt)=E^n$. The counterpart of the standard GKM theory for equivariant elliptic cohomology is
\be
E_A(\Fla)
=
\sqcup_{I\in \Il} \iota_IE_A(\pt) / \sim
\ee
where $\sim$ is induced by gluing $\iota_IE_A(\pt)=E^n$ and $\iota_{s_{a,b}(I)}E_A(\pt)=E^n$ along the diagonal $\Dl_{I,s_{a,b}(I)}=\{\xx\in E^n : x_a=x_b\}$, for any $I,a,b$, cf.~\cite{GKV}, \cite{GKM}, \cite[Example 4.4]{Ga}.
The isomorphism between the two descriptions of $E_A(\Fla)$ is induced by the map
\be
\sqcup_{I\in \Il} \iota_I(E_A(\pt)) \to
E^n \times_{E^{(n)}} (E^{(\la_1)} \times E^{(\la_2)} \times \ldots \times E^{(\la_N)})
\ee
whose restrictions to the copy $E^n$ labeled by $I$ is $\xx \mapsto (\xx, \xx_{I_1},\dots,\xx_{I_N})$.

\subsection{Cotangent bundle and the dynamical parameters}
Denote $T=A\times U(1)=U(1)^n\times U(1)$, and the extra $U(1)$ acts on
the fibers of $\tfl$ by multiplication. Since $\tfl$ is equivariantly homotopy
equivalent to $\Fla$ we have
\be
E_T(\tfl)=E_A(\Fla)\times E,
\ee
a scheme over $E_T(\mathrm{pt})=E^n\times E$.

As in \cite{AO} and \cite{FRV},
we consider an extended version of elliptic cohomology to accommodate for dynamical parameters
in quantum group theory
(or K\"ahler parameters in the terminology of \cite{AO}), namely
\be
\hE_T(\tfl):= E_T(\tfl)\times (\text{Pic}(\tfl)\otimes_\Z E) \simeq
E_T(\tfl)\times E^{N-1},
\ee
a scheme over $\hE(\pt)=E^{n+1+N-1}$, where the factors of $E^{N-1}$ correspond to the line bundles $\det(F_{p+1}/F_{p})$,
$p=1,\ldots,N-1$ over $\tfl$.

We will denote also by $\iota_I$ the map $\hE_T(\mathrm{pt})\to \hE_T(\tfl)$ induced by the inclusion of the fixed point $x_I$ into $\tfl$.
Then $\hE_T(\tfl)$ consists of the components $\iota_I\hE_T(\mathrm{pt})$ where $I$ runs over the elements of the set $\Il$.
By Section \ref {Cc}, we have a description of $\hE_T(\tfl)$ as a fiber product
\beq
\label{ghp}
\hE_T(\tfl)\,=\,\left(
E^n \times_{E^{(n)}} (E^{(\la_1)} \times E^{(\la_2)} \times \ldots \times E^{(\la_N)})\right)
\times E\times E^{N-1}.
\eeq
In particular we have the characteristic embedding
\beq
\label{he}
c\ :\ \hE_T(\tfl) \to
E^n \times E^{(\la_1)} \times E^{(\la_2)} \times \ldots \times E^{(\la_N)}\times E\times E^{N-1}
\eeq
of the extended $T$-equivariant cohomology scheme into a nonsingular projective variety.

\subsection{Admissible line bundles, and their sections}

We will consider sections of certain ``admissible'' line bundles over $\hE_T(\tfl)$.
These line bundles are, up to a twist by a fixed line bundle, those coming from the base scheme
$\Eh_{T}(\mathrm{pt})$. Let $p_T$ be the structure map
\beq
\label{st}
p_T\colon \Eh_T(\tfl)\to \Eh_T(\mathrm{pt})
\eeq
and
$\hat\chi = \chi \times \id\times\id$ the characteristic map
\beq
\hat\chi
\ :\ \hE_T(\tfl) \to
E^{(\la_1)} \times E^{(\la_2)} \times \ldots \times E^{(\la_N)}\times E\times E^{N-1}.
\eeq

Let $\bs v=(v^{(1)}, \dots, v^{(N)})$, $v^{(k)}=(v_{k,1},\dots,v_{k,\la_k})$ be coordinates on the universal cover
of $E^{\la_1}\times \ldots \times E^{\la_N}$ and let $y$, $\nub=(\nu_1,\dots,\nu_{N-1})$ be coordinates on
the universal cover of $E\times E^{N-1}$.

Let $M_\bla$ be the quadratic form defined in \Ref{mL}.
Clearly $M_{\bla}$ is symmetric under the action of the product $S_\bla=S_{\la_1}\times \dots\times S_{\la_N}$
of symmetric groups
permuting the second index of
$v_{l,a}$'s while the first index is fixed. Thus the line bundle $\LL(M_{\bla},0)$
on $E^{\la_1}\times \ldots \times E^{\la_N}\times E\times E^{N-1}$
can be considered as a line bundle on
$E^{(\la_1)}\times \ldots \times E^{(\la_N)}\times E\times E^{N-1}$.
\begin{defn} The {\em twisting line bundle} on $\hE_T(\tfl)$ is
$\mathcal T_\bla=\hat\chi^*\LL(M_\bla,0)$.
\end{defn}
\begin{defn} An {\em admissible line bundle} on
$\Eh_T(\tfl)$ is a line bundle of the form
\be
p_T^*\LL\otimes \mathcal T_{\bla},
\ee
for some line bundle $\LL$ on $\Eh_T(\mathrm{pt})$.
\end{defn}

We say that a meromorphic section on a complex
manifold {\em restricts to a meromorphic section} on a submanifold if
it is defined at its generic point, i.e., if the divisor of poles does
not contain a component of the submanifold.

We will consider meromorphic sections of line bundles on elliptic
cohomology schemes.
Recall that $\Eh_T(\tfl)$ has
components $Y_I=\iota_I\Eh_T(\mathrm{pt})$, corresponding to the
inclusion of the fixed points $x_I$, $I \in \Il$ into $\tfl$.

\begin{defn}\label{def-mero}
Let $\LL$ be a line bundle on $\Eh_T(\tfl)$. A {\em
meromorphic section} of $\LL$ is a collection of
meromorphic sections $s_I$ of $\LL|_{Y_I}$, labeled by
$I\in \Il$ and restricting to meromorphic sections
on all intersections $Y_{I_1}\cap\cdots\cap Y_{I_s}$ and such that
\be
s_I|_{Y_I\cap Y_J}=s_J|_{Y_I\cap Y_J},
\ee
for all $I,J$. A {\em holomorphic section} is a meromorphic section
whose restriction to each $Y_I$ is holomorphic.
\end{defn}

\subsection{$T$-equivariant elliptic cohomology classes}

\label{sec:ellHclasses}

\begin{defn}
Let $\LL$ be a line bundle on $\Eh_T(\on{pt})$.
A {\em
$T$-equivariant elliptic cohomology class} on $\tfl$ of degree
$\LL$ is a holomorphic section of the admissible line bundle
$p_T^*\LL\otimes \mathcal T_{\bla}$ on $\Eh_T(\tfl)$.
We denote by $ H_T^{\mathrm{ell}}(\tfl)_{\LL}$ the
complex vector
space of $T$-equivariant elliptic cohomology classes of degree
$\LL$ on $\tfl$.
\end{defn}

\subsection{Weight functions project to equivariant elliptic cohomology classes}
\label{sec:wecc}

In this section we show that the substitutions of weight functions $\Wt^{\>\on{ell}}_{\si\<,\>I}$
define $T$-equivariant elliptic cohomology classes on $\tfl$.

For $I\in\Il$ make the substitution
\bean
\label{subs}
&&
\phantom{aaa} q,\ h, \ \mu_1/\mu_2,\dots,\mu_{N-1}/\mu_N \
\mapsto \ e^{2\pi i\tau}, \ e^{2\pi i y},\ e^{2\pi i\nu_1},\dots, e^{2\pi i\nu_{N-1}},
\\
\notag
&&
t^{(k)}_1,\ldots,t^{(k)}_{\la^{(k)}}
\
\mapsto
\
e^{2\pi i v_{1,1}},\ldots,e^{2\pi iv_{1,\la_{1}}},\ \
e^{2\pi iv_{2,1}},\ldots,e^{2\pi iv_{2,\la_{2}}},\ \
\ldots,\ \
e^{2\pi iv_{k,1}},\ldots,e^{2\pi i v_{k,\la_{k}}},
\eean
for $k=1,\dots,N$ in the function $\Wt^{\>\on{ell}}_{\si\<,\>I}$. The resulting function
is denoted by $\WT^{\>\on{ell}}_{\si\<,\>I}(\bs v, y,\bs \nu)$.
The function $\WT^{\>\on{ell}}_{\si\<,\>I}$
can be viewed as a section of a certain line bundle on
$E^n\times E^{(\la_1)}\times\dots\times E^{(\la_N)} \times E\times E^{N-1}$.
Its pull-back $c^*\WT^{\>\on{ell}}_{\si\<,\>I}$ by the embedding
\be
c\ :\ \hE_T(\tfl) \to
E^n \times E^{(\la_1)} \times E^{(\la_2)} \times \ldots \times E^{(\la_N)}\times E\times E^{N-1}
\ee
is a section of the pull-back bundle and its
restriction to $Y_J=\iota_J\Eh_T(\mathrm{pt})$ is the evaluation of $c^*\WT^{\>\on{ell}}_{\si\<,\>I}$
at $v^{(k)}=\bs x_{J_k}$, $k=1,\dots, N$.

Recall the function $\psi_I(h,\mub)$ defined in \Ref{psI}. Make the substitution \Ref{subs} in
$\psi_I(h,\mub)$. Denote the resulting function
by $\psih_I(y, \nub)$.
The zeros of $\psih_I(y, \nub)$ define a divisor $D_I$ on $\hE_T(\mathrm{pt})$ consisting of hypertori
with equation of the form $\nu_a+\nu_{a+1}+\dots+\nu_b -jy$, $i\in\Z$.

Let $M_I(\xx,y,\nub)$ be the quadratic form defined in \Ref{mI}
and $\mc L_I=\mc L(M_I,0)$ the corresponding line bundle on $\hE_T(\mathrm{pt})$.

\begin{prop}
\label{prop-8}
The restriction $c^*\WT^{\>\on{ell}}_{\si\<,\>I}$
of $\WT^{\>\on{ell}}_{\si\<,\>I}$
to $\Eh_T(\tfl)$ is a meromorphic section of the admissible line
bundle $p_T^*\LL_{\si\<,\>I} \otimes \mathcal T_\bla$ for some line bundle $\LL_{\si\<,\>I}$
on $\hE(\mathrm{pt})$. Moreover, if $\si=\mathrm{id}$,
the restriction $c^*\WT^{\>\on{ell}}_{\mathrm{id},I}=c^*\WT^{\>\on{ell}}_I$
is a meromorphic section of the admissible line bundle
$p_T^*\LL_I\otimes \mathcal T_\bla$, where $\mc L_I$ is defined above,
and the restriction is a holomorphic section of
the line bundle $p_T^*\LL_I(D_I)\otimes \mathcal T_\bla$, where the notation
$\LL(D)$ means as usual the invertible sheaf of meromorphic sections of
a sheaf $\mc L$ whose poles are bounded by the divisor $D$.

\end{prop}

\begin{proof} The proof is similar to the proof of \cite[Proposition 5.9]{FRV}. It is enough to prove the statement for
$\si=\on{id}$.

The function $\hat G_I(\xx,y,\nub)\hat G_\bla(\bs v,y,\nub)/\Eh^{\>\on{ell}}_\bla(\bs v, y)$ defines a meromorphic section
of the line bundle $\mc L(M_I+M_\bla,0)$ on
$E^n\times E^{\la_1}\times\dots\times E^{\la_N} \times E\times E^{N-1}$. That line bundle lifts to the line bundle
$p_T^*\LL_I\otimes \mathcal T_\bla$ on $\Eh_T(\tfl)$.
Lemmas \ref{lem:tr} and \ref{lem:div_e} imply that $c^*\WT^{\>\on{ell}}_I$ is a meromorphic section of
the admissible line bundle $p_T^*\LL_I\otimes \mathcal T_\bla$ on $\Eh_T(\tfl)$.
The fact that $c^*\WT^{\>\on{ell}}_I$ is a holomorphic section of
the line bundle $p_T^*\LL_I(D_I)\otimes \mathcal T_\bla$ follows from the formula for the weight functions.
\end{proof}

\begin{cor}
\label{cor:W-coh}
For any $\si\in S_n$ and $I\in\Il$ the restriction
$c^*\WT^{\>\on{ell}}_{\si\<,\>I}$ defines a $T$-equivariant elliptic cohomology class on $\tfl$.
\end{cor}

\subsection{Elliptic stable envelope for cotangent bundles of partial flag varieties}

In this section we introduce a version of elliptic stable envelopes.
Our elliptic stable envelopes are defined in terms of the elliptic weight
functions. In Theorem \ref{thm:ax} we give their axiomatic definition in the spirit
of \cite{MO1}, \cite{FRV}. {\it It would be interesting to understand the
relation of our definition with the one sketched in \cite{AO}.}

For $\si\in S_N$, $I\in \Il$ we call the $T$-equivariant elliptic cohomology class
$c^*\WT^{\>\on{ell}}_{\si\<,\>I}$ the {\it
stable envelope associated with $\si$ and the fixed point} $x_I\in\tfl$.

Consider $\C^N$ with basis $v_1,\dots,v_N$. The standard basis of $(\C^N)^{\ox n}$ is formed by the vectors
$v_I$ labeled by partitions $I=(I_1,\dots,I_N)$ of $\{1,\dots, n\}$,
$v_I=v_{i_1}\otimes \dots\ox v_{i_n}$,
where $i_a=j$ if $a\in I_j$.

For $\id\in S_n$ we may consider the map
\beq
\label{staB}
\Stab_\id : (\C^N)^{\ox n} \to \oplus_{\bla\in (\Z_{\geq 0})^N,\, |\bla|=n}\oplus_{I\in\Il} H^{\>\on{ell}}_T(\tfl)_{\mc L_I(D_I)},
\eeq
sending $v_I$ to the cohomology class $c^*\WT^{\>\on{ell}}_I$.

Similarly we may define the maps $\Stab_\si$ for all $\si\in S_n$. The maps
$\Stab_\si, \Stab_{\si'}$ are related
by the elliptic dynamical $R$-matrix, see Theorem~\ref{thm:recur}, cf.~\cite[Theorem 7.1]{RTV2}
and Section~\ref{NoR}.

\begin{rem}
For $\si\in S_n$ the collection $(c^*\WT^{\>\on{ell}}_{\si\<,\>I})_{I\in\Il}$ forms a basis of the $T$-equivariant
cohomology of $\tfl$ in the sense of \cite[Theorem 5.23]{FRV}.
We will discuss this fact in the next paper.
\end{rem}

\begin{rem}
Using the map $\Stab_\id$ one can construct an action of the dynamical elliptic quantum groups associated with
$\frak{gl}_N$ on the extended equivariant cohomology
\\
$\sqcup_{\bla\in \Z_{\geq 0}^N,\, |\bla|=n}\hE_T(\tfl)$.
The action is by $S_n$-equivariant admissible difference operators acting on sections of admissible line bundles, see the case
$N=2$ in \cite{FRV}. See similar constructions for equivariant cohomology and $K\<$-theory in \cite{GRTV, RTV2,RV1}.
We plan to discuss this action in the next paper.
\end{rem}

\begin{rem}
The equivariant elliptic cohomology class $c^*\WT^{\on{ell}}_I$ has analogs in equivariant cohomology
and equivariant $K\<$-theory of $\tfl$, see \cite{GRTV,RTV1,RTV2}.
The analog of $c^*\WT^{\on{ell}}_I$ in equivariant cohomology is
the equivariant Chern--Schwartz--MacPherson class (or characteristic cycle)
of the open Schubert variety $\Om_I$, see \cite{RV2}.
Hence $c^*\WT^{\on{ell}}_I$ may be viewed as an equivariant elliptic version
of the Chern--Schwartz--MacPherson class.
\end{rem}

\subsection{Axiomatic definition of the stable envelope}

\begin{thm}
\label{thm:ax}
For any $I$ the $T$-equivariant elliptic cohomology class $c^*\Wt^{\>\on{ell}}_I$ satisfies the following properties.
\begin{itemize}
\item[(i)] It is a meromorphic section of the admissible line bundle $p_T^*\LL(N_I,0)\otimes \mathcal T_{\bla}$
for a suitable quadratic form $N_I$.

\item[(ii)] The restriction of $c^*\Wt^{\>\on{ell}}_I$ to the component $Y_I$, written as a function $\C^{n+1+N-1}\to \C$
with transformation properties determined by the line bundle $p_T^*\LL(N_I,0)\otimes \mathcal T_{\bla}$,
equals
\beq
\label{Pie}
\Phh^{\>\on{ell}}_I(\xx,y)\,=
\,\prod_{k<l}\,\prod_{a\in I_k}\!
\biggl(\>\prod_{\satop{b\in I_l}{b<a}}\!\theta(y+x_b-x_a)\!
\prod_{\satop{b\in I_l}{b>a}}\!\theta(x_b-x_a)\biggr).
\eeq

\item[(iii)] The restriction of $c^*\Wt^{\>\on{ell}}_I$ to any component $Y_J$, written as a function $\C^{n+1+N-1}\to \C$ with transformation properties determined by $p_T^*\LL(N_I,0)\otimes \mathcal T_{\bla}$, is of the form
\beq\label{eqn:horizontal}
\frac 1{\psih_I(y,\nub)}
\prod_{k<l}\prod_{a\in J_k}\prod_{\satop{b\in J_l}{b<a}}
\theta(y+x_b-x_a) \cdot F_{I\<,\:J},
\eeq
where $F_{I\<,\:J}$ is a {\em holomorphic} function.
\end{itemize}
Moreover, these three properties uniquely determine the $T$-equivariant elliptic cohomology class $c^*\Wt^{\>\on{ell}}_I$.
\end{thm}

\begin{proof}
The proof is analogous to the proof of \cite[Theorem A.1]{FRV}.

Properties (i-iii) of $c^*\Wt^{\>\on{ell}}_I$ follow from Lemmas
\ref{lem:expec} and \ref{lem:div_e}. Now we prove that properties (i-iii) uniquely determine $c^*\Wt^{\>\on{ell}}_I$. Let $s$ be
a section with properties (i-iii) for some admissible line bundle $p_T^*\LL(N_I,0)\otimes \mathcal T_{\bla}$.
Property (ii) implies that $N_I=M_I$, where $M_I$ is given by \Ref{mI}. Hence $\LL(N_I,0)=\mc L_I$.

Denote by $\ka_I$ the difference between $s$ and
$c^*\Wt^{\>\on{ell}}_I$. Assume that the difference is nonzero.
Then there exists a $J$ such that $\kappa_I$ restricted to $Y_J$ is not 0. For a total ordering $\prec$ refining the partial order $<$ on $\Il$ let us choose $J$ to be the largest index with the property $\kappa_I|_{Y_J}\ne0$. We have $J\ne I$ because of the second property.

We claim that $\kappa_I|_{Y_J}$, written as a function $\C^{n+1+N-1}\to \C$, with transformation properties determined by $p_T^*\LL_I\otimes \mathcal T_\bla$, is of the form
\beq\label{eqn:div1}
\frac 1{\psih_I(y,\nub)}
\prod_{k<l}\,\prod_{a\in J_k}\!
\biggl(\>\prod_{\satop{b\in J_l}{b<a}}\!\theta(y+x_b-x_a)\!
\prod_{\satop{b\in J_l}{b>a}}\!\theta(x_b-x_a)\biggr)
\cdot
F_1,
\eeq
where $F_1$ is holomorphic. The presence of the first theta-factors
$\theta(y+x_b-x_a)$ in \Ref{eqn:div1} follows from property (iii).
The presence of the second theta-factors $\theta(x_b-x_a)$
in \Ref{eqn:div1} follows from the fact that $J$ is the largest index
such that $\kappa_I|_{Y_J}\ne0$, cf.~the proof of \cite[Theorem A.1]{FRV}.

Observe that the product of theta functions in (\ref{eqn:div1}) equals
$c^*\Wt^{\>\on{ell}}_{J}|_J$ by property (ii). Hence $F_1/\psih_I$
is a meromorphic section of the line bundle $\mc L(M_I-M_J,0)$.

Let $a\in \{1,\dots,n\}$ and $k,l\in\{1,\dots,N\}$, $k\ne l$, be such that $a\in I_k\cap J_l$. Consider
$F_1/ \psih_I$ as a function of $x_a$. Denote it by $f(x_a)$. The function $f(x_a)$ is a holomorphic function of $x_a$ for
generic fixed other arguments. Comparing the $x_a$-dependence of the quadratic form $M_I-M_J$
we obtain that
\beq\label{eqn:qq}
f(x_a+\tau)=e^{-2\pi i (2(\nu_l+\dots+\nu_{N-1}-\nu_k-\dots-\nu_{N-1})+...)} f(x_a), \qquad f(x_a+1)=f(x_a),
\eeq
where the dots indicate the terms independent on $x_a$ and $\nub$. Using the 1-periodicity, we expand
$f(x_a)=\sum_{m\in \mathbb Z} a_m e^{2\pi i m x_a}$,
and using the first transformation property of (\ref{eqn:qq}) we obtain that $a_m=0$ for all $m\in \mathbb Z$. Hence
$F_1=0$, and in turn, $\kappa_I|_{Y_J}=0$. This is a contradiction proving that $\kappa_I$ is 0 on all $Y_J$.
\end{proof}

\begin{rem}
The classes appearing in the second and third properties
in Theorem \ref{thm:ax}
can be identified as horizontal and vertical parts of the equivariant elliptic normal Euler classes of $C\Om_{\si\<,\>I}$ near the fixed point $x_I$, cf.~Section \ref{sec:edefs}.
\end{rem}

\begin{rem}
The stable envelopes are defined in equivariant cohomology, equivariant $K\<$-theory, and equivariant elliptic cohomology, see \cite{MO1,MO2}, for the cotangent bundles of partial flag varieties the definitions are discussed in \cite{RTV1, RTV2, FRV} and in
Sections \ref{sec:5}, \ref{sec7}.
In all three cases the definition consists of three axioms. One axiom says
that the restriction $\ka_J|_{x_J}$ of the stable envelope $\ka_J$
to the fixed
point $x_J$ equals some ``expected'' product. Another axiom says that the restriction
$\ka_J|_{x_I}$ of $\ka_J$
to any fixed point $x_I$ should be divisible by some product determined by the point $x_I$.
This is a support type axiom, see Section \ref{sec:5}.
The last axiom says that the restriction $\ka_J|_{x_I}$ should be ``smaller'' than the restriction $\ka_I|_{x_I}$.
The products in the first two axioms are the products of linear function or trigonometric functions or theta functions with the same arguments when we change from equivariant cohomology to equivariant $K\<$-theory and then to equivariant elliptic cohomology.
It is interesting to see how the notion of ``smallness'' is changing in these three cases.
In the cohomology case the restrictions are polynomials and one polynomial is smaller than another if the degree of the first is smaller than the degree of the second. In the $K\<$-theory case the restrictions are Laurent polynomials and one Laurent polynomial is smaller than another
if the Newton polytope of the first can be parallelly moved inside the Newton polytope of the second. In the elliptic case the smallness is the requirement to be a section of an admissible line bundle, in other words, to be an equivariant elliptic cohomology class, see condition (i) in Theorem \ref{thm:ax}. In other words, the most nontrivial condition of smallness just dissolves in the definition of a cohomology class.

\end{rem}

\section{Appendix: Comparison of Bethe algebras}

The Faddeev-Takhtajan-Reshetikhin formalism applied to the $R$-matrix $R(z,h)$ given by \Ref{RO}
produces the quantum loop algebra $\Uen$, the evaluation $N$-dimensional representations of $\Uen$,
and commutative Bethe subalgebras
$\Bck$ of $\Uen$ depending on complex parameters $q=(q_1,\dots,q_N)$ called the quantum parameters,
see explicitly these constructions in
\cite[Sections 10 and 11]{RTV2}. The Bethe algebra acts on the tensor product of $n$ evaluation representations and preserves
the weight decomposition
$(\C^N)^{\otimes n}=\oplus_{\bla\in\Z^N_{\geq 0},\,|\bla|=n} (\C^N)^{\otimes n}_\bla$.
The image of the Bethe algebra in the algebra of endomorphisms of a weight subspace $(\C^N)^{\otimes n}_\bla$
is described
in \cite[Theorem 13.3]{RTV2} by generators and relations
in terms of a discrete Wronski map depending on parameters $q$.

Let $\Rc(z,h)$ be defined by \Ref{RN} for the anti-dominant alcove $\Delta$,
i.e.~with all $m_{j,k}$ equal to zero. The matrix $\Rc(z,h)$ differs from
the matrix $R(z,h)$ by change in the normalization convention involving $h$.
The Faddeev-Takhtajan-Reshetikhin formalism applied to $\Rc(z,h)$
produces the new quantum loop algebra $\Uen'$
isomorphic to $\Uen$, new evaluation representations,
and new commutative Bethe subalgebras $\Bckp\!\subset \Uen'$.
It turns out that {\it the image of the new Bethe algebra $\Bckp$ with parameters $q'=(q'_1,\dots,q'_N)$
in the algebra of endomorphisms of $(\C^N)^{\otimes n}_\bla$
coincides with the image of the original Bethe algebra $\Bck$ with quantum parameters
$q=(q_1,\dots,q_N)$ if}
\be
q_k'=\>h^{-\la_1-\dots-\la_{k-1}+\la_{k+1}+\dots+\la_N}q_k,
\qquad k=1,\dots,N.
\ee
Hence the image of the new Bethe algebra is also described by the discrete Wronski map.
We will discuss this fact in detail in the next paper.

The importance of that statement lies in the following. In \cite{OS} Okounkov
and Smirnov consider the equivariant $K\<$-theory algebra
of a Nakajima variety and the stable envelopes of the anti-dominant alcove.
Using the $R$-matrices of that alcove they define the associated
quantum loop algebra with the action on the equivariant $K\<$-theory algebra.
Conjecturally, the associated Bethe algebra of that action is the algebra of quantum multiplication on the equivariant
$K\<$-theory algebra, cf. \cite{OS,RTV2,PSZ}. For the cotangent bundle $\tfl$ of a partial flag variety $\F_\bla$
the constructions of this paper and \cite{OS,RTV2}
identify the Bethe algebra of the cotangent bundle with the image of the Bethe algebra $\Bckp$ in the endomorphism algebra of $(\C^N)^{\otimes n}_\bla$. Hence the formulated statement gives a conjectural description
of the algebra of quantum multiplication on the equivariant $K\<$-theory algebra of $\tfl$
by generators and relations in terms of the discrete Wronski map.

\end{document}